\documentclass[13pt]{amsart}
\usepackage{amsfonts}
\usepackage{amsmath}
\usepackage{amsmath,amssymb,amsthm,amsfonts,amscd}
\usepackage[mathscr]{eucal}
\usepackage{graphicx}
\usepackage{subfigure}
\usepackage{float}
\usepackage{color}
\newtheorem{theorem}{Theorem}[section]
\newtheorem{lemma}{Lemma}[section]

\newtheorem{definition}{Definition}[section]
\newtheorem{example}{Example}[section]

\theoremstyle{remark}

\newtheorem{proposition}{Proposition}[section]

\title[On complete cscK metrics with PMY asymptotic property]{On complete constant scalar curvature K\"ahler metrics with Poincar\'e-Mok-Yau asymptotic property}
\author{Jixiang Fu}
\author{Shing-Tung Yau}
\author{Wubin Zhou}
\address{Institute of Mathematics\\ Fudan University \\ Shanghai
200433, China} \email{majxfu@fudan.edu.cn}
\address{Department of
Mathematics\\ Harvard University\\Cambridge, MA 02138, USA}
\email{yau@math.harvard.edu}
\address{Shanghai Center for Mathematical Sciences\\
Shanghai 200433, China} \email{wubin\_zhou@fudan.edu.cn}

\begin{document}

\begin{abstract}
Let $X$ be a compact K\"ahler manifold and $S$  a subvariety of $X$ with higher  co-dimension. The aim is to study complete constant scalar curvature K\"ahler metrics on non-compact K\"ahler manifold $X-S$ with  Poincar\'e--Mok--Yau asymptotic property (see Definition \ref{def}). In this paper, the methods of Calabi's ansatz and the moment construction are used to provide  some special examples  of such metrics.

\end{abstract}
\maketitle
\section{Introduction}
In K\"ahler geometry, a basic question is to find on a K\"ahler manifold a
canonical metric in each K\"ahler class, such as a K\"ahler--Einstein (K--E) metric,
a constant scalar curvature K\"ahler (cscK) metric, or even an extremal metric.
If $X $ is a compact K\"ahler manifold with the definite first Chern class, the question has been solved thoroughly and there
are lots of references on this topic. Among these, the fundamental one \cite{Yau} is
on the Calabi conjecture solved by Yau.

In the non--compact case, Tian and Yau proved in \cite{TY1,TY2} that there exists a
complete Ricci--flat metric on $X^\ast = X-D$, where $X$ is a compact K\"ahler
manifold and $D$ is a neat and almost ample smooth divisor on $X$; or $X$ is a compact
K\"ahler orbifold and $D$ is a neat, almost ample and admissible divisor on $X$.

Several years ago, the second named author presented the following question:
\vspace{2mm}

\noindent \textbf{Problem 1}. Assume that $X$ is a compact K\"ahler manifold and $S$ is its higher co-dimensional subvariety. Let $X^\ast =  X-S$. How
to find a complete canonical metric on such a non-compact K\"ahler manifold
$X^\ast$?
\vspace{2mm}

Certainly,  this problem is equivalent to finding a canonical metric on $\bar X-D$, where $\bar X$ is a compact K\"ahler manifold and $D$ is a divisor on $\bar X$. More precisely,
blowing up of $X$ along $S$, one obtains a new compact  K\"ahler manifold $\bar X=Bl_S(X)$. Then $X^\ast$ is bi-holomorphic to $\bar X-D$ where $D$ is the exceptional divisor of this blow-up. Hence our problem is transferred to finding a complete canonical metric on $\bar X-D$. However,  this blowing up process can not make  Problem 1 easier since it does not alter the geometric properties of $X^\ast$. For example, although $\mathbb CP^2-p$ is bi--holomorphic to $Bl_p(\mathbb CP^2)-D$, we can not use Tian--Yau's results mentioned above to get a complete K--E metrics on $\mathbb CP^2-p$ since the exceptional divisor $D$ is not ample.

The basic strategy to solve Problem 1 is to perturb one family of approximate metrics on $X^\ast$. This method has been carried out successfully in  \cite{AP1,AP2,APS,Ga,Ga2} to construct cscK or  extremal metrics on blow-up of a K\"ahler manifold at some points.
The key point is that the csck metric of Burns-Simanca \cite{Si}  on $Bl_0(\mathbb C^n)$
is asymptotic locally Euclidean (ALE) at infinity.

Motivated by this, if one  want to solve Problem 1 on $M-\{p_1,\cdots,p_l\}$, one should first construct a canonical metric on $\mathbb C^n-0$ which is also ALE at infinity. Fortunately, the metrics in the following theorem admit this asymptotic property. Let $r^2$ be the Euclidean norm squared function on $\mathbb C^n$.

\begin{theorem}\label{thm1}
There exist on $\mathbb C^n-0$ a family of complete  zero scalar curvature K\"ahler metrics $\eta_a=\sqrt{-1}\partial\bar\partial u_a(r^2)$ ($a>0$) with the following asymptotic properties: As $r^2\to 0$,
  \begin{equation*}
   u_a(r^2)= a\log r^2-\frac{2a}{n(n-1)}\log(-\log r^2)+O((\log r^2)^{-1});
 \end{equation*}
And as $r^2\to \infty$,
\begin{displaymath}
u_a(r^2)=\left\{
\begin{array}{ll}
 r^2+2a\log r^2+\frac{a^2}{2r^2}+O(\frac{1}{r^4}),& \text{for }\ \  n=2,\\
 r^2-\frac{na^{n-1}}{(n-1)(n-2)}(r^2)^{2-n}+\frac{a^n}{n}(r^2)^{1-n}+O((r^2)^{-n}) & \text{for }\ \  n\geq 3.
 \end{array}\right.
\end{displaymath}
Here $O(h(r^2))$ is a smooth function whose $k-$th partial derivatives for all $k\geq 0$ are bounded by a constant times $|\partial^k h(r^2)| $.
\end{theorem}

For the cases of  constant scalar curvature $c\neq 0$, we have the following theorem. Denote $D^n$ as the unit disc of $\mathbb C^n$.

\begin{theorem}\label{thm2}
1. For any $c<0$,  there exist on $D^n-0$ a family of complete K\"ahler metrics $\sqrt{-1}\partial\bar\partial u_a(r^2)$ ($a>0$) with constant  scalar curvature $c$.  As $r^2\to 1$,
  these metrics are asymptotic to the Poincar\'e metric
  $$\sqrt{-1}\frac{n(n+1)}{-c}\partial\bar\partial\log(1-r^2).$$

 2. For $c>0$ and $a>0$ with $ac<n(n-1)$, there exists on $\mathbb C^n-0$ a K\"ahler  metric $\sqrt{-1}\partial\bar\partial u_a(r^2)$ with  constant scalar curvature $c$  which are not complete at infinity and asymptotic to
$$\sqrt{-1}\partial\bar\partial(b\log r^2+\kappa r^{-\frac{ 2}{\kappa}})$$
for two constants $b(>a)$ and $\kappa(>1)$.

In both cases, the metrics have the following asymptotic property: As $r^2\to 0$,
\begin{equation*}
u_a(r^2)=a\log r^2-\frac{2a}{n(n-1)-ac}\log(-\log r^2)+ O((\log r^2)^{-1}).
\end{equation*}
\end{theorem}

Naturally one would ask whether there are any complete K--E metrics on $\mathbb C^n-0$ or on  $D^{n}-0$.
 Using the method in  \cite{MY}, we can prove the following theorem which turns out that in some sense the choice of cscK metrics is optimal.

\begin{theorem}[\cite{MY}]\label{thm3}
  There do not exist any complete K\"ahler-Einstein metrics on $\mathbb C^n-0$ or on $D^n-0$.
\end{theorem}
 In fact,  we can prove $X^*$ can not admit any complete K\"ahler-Einstein metrics in our further paper \cite{FYZ2}.
Theorems \ref{thm1} and  \ref{thm2} remind us to recall the Mok--Yau metric in \cite{MY}. In 1980s, Mok and Yau introduced on $D^{n}-0$ the metric with bounded Ricci curvature
$$\sqrt{-1}\partial\bar\partial (\log r^2-\log(-\log r^2)).$$
They used this metric to characterize domains of holomorphy by holomorphic sectional conditions. Comparing the Mok--Yau metric with the metrics in Theorems \ref{thm1} and \ref{thm2} leads to the following definition.

\begin{definition}[see also \cite{FYZ1}]\label{def}
  Let $X$ be a compact K\"ahler manifold with a K\"ahler metric $\omega_X$ and let $S$ be a higher co--dimensional subvariety.
  A K\"ahler metric $\omega$ on $X-S$ has the Poincar\'e--Mok--Yau (PMY) asymptotic property if near the subvariety $S$
  $$\omega=\omega_X+\sqrt{-1}\partial\bar\partial(a\log r^2-b\log(-\log r^2)+ O((\log r^2)^{-1})),$$
  where $r$ is some distance function to $S$, and $a$ and $b$ are two positive constants.
\end{definition}

In the second part of this paper, we generalize Theorems \ref{thm1} and \ref{thm2}  to the cases of holomorphic vector bundles.  We will use the moment construction to find  complete cscK metrics on the complement of the zero section in (the total space of) a holomorphic vector bundle or a projective bundle (i.e. a ruled manifold). There are many references such as \cite{ACGT,HS,KS,PP} which use the method of moment construction to look for canonical metrics on K\"ahler manifolds. One can consult \cite{HS} for construction of cscK metrics on vector bundles and  \cite{ACGT} for extremal metrics on ruled manifolds.

 Let $M$ be a compact $m$--dimensional K\"ahler manifold with a cscK metric $\omega_M$. Let $(L,h)$ be a holomorphic line bundle over $M$ with a hermitian metric $h$, which is given by  local positive functions $h(z)$ defined on the open sets which  locally trivialize  $L$. For the technical  reason, assume that  there exists a constant $\lambda$ such that
\begin{equation}\label{lo}
\sqrt{-1}\partial\bar\partial \log h(z)=\lambda \omega_M.
\end{equation}
Let $(E,\pi) $ be the direct sum of $n\ (\geq 2)$ copies of  $L$  with associate hermitian metric $h$. Let $\nu$ be the logarithm of the fibre norm squared  function defined by $h$ and consider Calabi's ansate
\begin{equation*}
\omega=\pi^\ast\omega_M+\sqrt{-1}\partial\bar\partial f(\nu).
\end{equation*}
Denote the zero section of $E$ simply by $M$.  Also denote $\mathbb U$ as the set of points $p$ in $E$ such that $\nu(p)<0$. We first concern about csck metrics with {\sl PMY asymptotic property} on $E-M$ or on $\mathbb U-M$.

\begin{theorem}\label{thmlag0}
Let $M$ be a compact K\"ahler manifold and $\omega_M$ a K\"aler metric with constant scalar curvature $c_M$. Let $L$ be a holomorphic line bundle over $M$ with a hermitian $h$. Assume that $h$ and $\omega_M$  satisfy (\ref{lo}). Let $E$ be the direct sum of $n\ (\geq 2)$ copies of $L$.

1. If $\lambda\geq 0$, there exists a constant $c_0$ such that for any $c\leq c_0$, there exists  on $\mathbb U-M$ or on  $E-M$ a complete K\"ahler metric with  constant scalar curvature $c$.
Such metrics admit the Poincar\'e--Mok--Yau asymptotic property except the case that the metrics are defined on $\mathbb U-M$ with $c=c_0 (<0)$ and $\lambda>0$.

2.  If $\lambda<0$ and $c_M>0$, there exists on $E-M$ a complete positive constant scalar curvature K\"ahler metric with Poincar\'e--Mok--Yau asymptotic property.
\end{theorem}

If $\lambda=0$, this theorem generalizes  Theorem \ref{thm1} and the case $c<0$ of Theorem \ref{thm2}.

We then consider Problem 1 on a projective bundle.
Denote $\mathcal O$ as the structure sheaf of $M$. The projective bundle $\mathbb P(E\oplus\mathcal O)$ over  $M$ has a globally defined section $s$: for $q\in M$, $s(q)$ is a point corresponding to the line $\mathcal O_q$.
The following theorem gives some special solutions to Problem 1.

\begin{theorem}\label{thmpm}
Under the assumptions of Theorem \ref{thmlag0}, if $\lambda<0$ and $c_M\in\mathbb R$ or if $\lambda >0$ and $c_M\in (m(m+2n-1)\lambda,+\infty)$, there exists on $\mathbb P(E\oplus \mathcal O)-M$ a complete constant scalar curvature K\"ahler metric with Poincar\'e--Mok--Yau asymptotic property.
\end{theorem}


{\bf Acknowledgement.} Fu is supported in part by NSFC grant 11421061. Yau is supported in part by NSF grant DMS-0804454.  Zhou is supported by China Postdoctoral Science Foundation funded project grant No.2015M571479 .

\section{Complete cscK metrics on $\mathbb C^n-0$  and $D^n-0$.}\label{sec2}
In this section, we construct complete cscK metrics on $\mathbb C^{n\ast}$ or $D^{n\ast}$. Here  denote $\mathbb C^{n\ast}=\mathbb C^n-0$ and $D^{n\ast}=D^n-0$.  We  first follow Calabi's method \cite{Ca2} to get an ODE on the K\"ahler potential. Then we determine the constants of integration appeared in the ODE by discussing  completeness of the metric near the punctured point. Afterwards we analyze the asymptotic properties of the K\"ahler potential. Thus, Theorems \ref{thm1} and \ref{thm2} are proven. In the last subsection, we give some remarks and a simple proof of Theorem \ref{thm3}.

\subsection{Calabi's ansatz.}
Let $w=(w_1,w_2,\cdots,w_n)$ be the coordinates of $\mathbb C^n$.  Assume that the K\"ahler metric we are seeking for is rotationally symmetric. That is, if we  let
$$r^2=\sum_{\alpha=1}^n| w_\alpha| ^2\  \  \ \textup{and}\ \ \ t=\log r^2,$$
then the K\"ahler potential is a function $u(t)$. By a direct calculation,
\begin{equation*}
  g_{\alpha\bar\beta}:=\frac{\partial^2 u(t)}{\partial w_\alpha\partial\bar w_\beta}=e^{-t}u'(t)\delta_{\alpha\beta}+e^{-2t}\bar w_\alpha w_\beta
  (u''(t)-u'(t)).
\end{equation*}
Hence,
 $$\det(g_{\alpha\bar\beta})=e^{-nt}u'(t)^{n-1}u''(t),$$
 and $\omega=\sqrt{-1}\partial\bar\partial u(t)$ is a K\"ahler metric if and only if
\begin{equation*}\label{pos}
u'(t)>0\ \ \ \textup{and} \ \ \ u''(t)>0.
\end{equation*}
 For simplicity, let
\begin{equation}\label{v}
  v(t)=-\log\det(g_{\alpha\bar\beta})=nt-(n-1)\log u'(t)-\log u''(t).
\end{equation}
The components of the Ricci tensor of $\omega$ are
$$
  R_{\alpha\bar\beta}=\frac{\partial^2v(t)}{\partial w_\alpha\partial\bar w_\beta}=e^{-t}v'(t)\delta_{\alpha\beta}+e^{-2t}\bar w_\alpha w_\beta (v''(t)-v'(t))
$$
and then the scalar curvature is
\begin{equation}\label{c(t)}
  c(t)=g^{\alpha\bar\beta}R_{\alpha\bar\beta}=(n-1)\frac{v'(t)}{u'(t)}
  +\frac{v''(t)}{u''(t)}.
\end{equation}
Here $(g^{\alpha\bar\beta})$ denotes  the inverse matrix of
$(g_{\alpha\bar\beta})$. Explicitly,
\begin{equation*}
  g^{\alpha\bar\beta}=\frac{e^t}{u'(t)}\delta_{\alpha\beta}+w_\alpha\bar w_\beta(\frac{1}{u''(t)}-\frac{1}{u'(t)}).
\end{equation*}

Assume that the scalar curvature of  $\omega$ is a constant $c$. Integrating (\ref{c(t)}) with the integrating factor $u'(t)^{n-1}v'(t)$, we obtain the first order differential relation between $u(t)$ and $v(t)$
$$
  v'(t)u'(t)^{n-1}=\frac{1}{n}c (u'(t))^n +c_1
$$
with an arbitrary constant $c_1$. Substituting (\ref{v}) to the above equation and multiplying  both sides with $u''(t)$,  we get the equation
$$nu'(t)^{n-1}u''(t)-(n-1)u'(t)^{n-2}u''(t)^2-u'(t)^{n-1}u'''(t)
=\frac{1}{n}c u'(t)^n u''(t)+c_1u''(t).$$Integrating the above
equation, we obtain
$$u'(t)^n-u'(t)^{n-1}u''(t)=\frac{c}{n(n+1)} u'(t)^{n+1}+c_1u'(t)+c_2$$
with another arbitrary constant $c_2$.  If we denote $\phi(t)=u'(t)$, then
the above equation can be written as the first order differential
equation
\begin{equation}\label{phia}
  \frac{d\phi}{dt}=\frac{F(\phi)}{\phi^{n-1}}
\end{equation}
with
\begin{equation*}
F(\phi)=-\frac{c}{n(n+1)}\phi^{n+1} +\phi^n -c_1\phi-c_2
\end{equation*}
or rewritten as
\begin{equation}\label{dt}
  dt=\frac{\phi^{n-1}d\phi}{F(\phi)}.
\end{equation}
It follows  that $u(t)$ is a K\"ahler potential if and only if
\begin{equation*}
  \phi(t)>0\qquad \text{and}\qquad F(\phi)>0.
\end{equation*}

\subsection{Completeness.}
Assume that $\phi=\phi(t)$ is a solution to ODE (\ref{phia}) or (\ref{dt}) in an interval $(-\infty,t_0)$, where $t_0$ can be equal to $+\infty$, and assume that it determines a K\"ahler potential $u=u(t)$ in the  punctured disc  $D^{n\ast}(r_0)$ with radius $r_0=\exp(\frac {t_0}{2})$. In this subsection, for the sake of the completeness of $\omega=i\partial\bar\partial u(t)$ near the punctured point, the constants $c_1$ and $c_2$ in ODE (\ref{phia}) can be determined.
The key point is the following observation.
\begin{lemma}\label{lemC}
Under the above assumption, the metric $\omega$ determined by the K\"ahler potential   $u(t)$  is complete near the punctured point if and only if  $F(\phi)$  has a factor  $(\phi-a)^2$ with $a>0$ and $\lim_{t\to -\infty} \phi=a$. Hence
\begin{equation}\label{629}
c_1=na^{n-1}-\frac{c}{n}a^n\qquad  \text{and} \qquad c_2=(1-n)a^n+\frac{c}{n+1}a^{n+1}.
\end{equation}
\end{lemma}
\begin{proof}
Since the metric $\omega=\sqrt{-1}\partial\bar\partial u(t)$ is rotationally symmetric, for any point $p\in D^{n\ast}(r_0)$, the ray $\gamma(s)=sp$, $s\in(0,1]$, is a geodesic. The tangent vector of this curve at the point $sp$ is $p$ and its square norm under the metric $\omega$ is, if we assume that $r^2(p)=1$,
$$|p|^2_{sp}=\sum w_\alpha(p)\bar w_\beta(p)g_{\alpha\bar\beta}(sp)=u''(t)r^{-2}.$$
Then the length of $\gamma(s)$ is
$$l=\int^1_0|\gamma'(s)|_{sp}ds=\int_0^1\sqrt{u''(t)}\frac{dr}{r}=\frac 1 2 \int^0_{-\infty}\sqrt{u''(t)}dt=\frac 1 2 \int_{-\infty}^0\sqrt{\frac{d\phi}{dt}}dt.$$
Under the  assumption that $d\phi/dt>0$ and $\phi(t)>0$, there is a nonnegative constant $a$ such that
  $$\lim_{t\to -\infty}\phi=a.$$
By equation (\ref{dt}), we have
$$l=\frac 1 2\int_{a}^{\phi(0)}\sqrt{\frac{\phi^{n-1}}{F(\phi)}}d\phi.$$
The completeness requires $l=+\infty$, which is equivalent to the fact  that $F(\phi)$ has a factor $(\phi-a)^2$.  Hence, we can determine $c_1$ and $c_2$ as in (\ref{629}).

We claim $a>0$. If $a=0$, $F(\phi)=\phi^n(-\frac{c\phi}{n(n+1)}+1)$. Hence,
$l<+\infty$, which leads to a contradiction.
\end{proof}

\subsection{Discussions of the solutions and proofs of Theorems \ref{thm1} and \ref{thm2}.}
Because of completeness, in this subsection  assume that the constants $c_1$ and $c_2$ have been chosen as in (\ref{629}). Hence, $F(a)=F'(a)=0$. Since
\begin{equation*}
F''(\phi)=(n(n-1)-c\phi)\phi^{n-2},
\end{equation*}
in case $c\leq 0$,  $F''(\phi)>0$ and so $F(\phi)>0$ on domain $(a,+\infty)$. In case  $c>0$, if assume that the constants $a$ and $c$ satisfy
\begin{equation*}
ac<n(n-1),
\end{equation*}
then on domain $(a,\frac{n(n-1)}{c})$, $F''(\phi)>0$ and so $F(\phi)>0$. Hence,  $F(\phi)>0$ on domain $(a,b)$ for some constant $b$ or $b=+\infty$. Thus, we obtain the solution of equation (\ref{dt}), up to a constant:
\begin{equation}\label{sol1}
t=t(\phi)=\int_{\phi_0}^\phi\frac{x^{n-1}}{F(x)}dx,\qquad \phi\in (a,b),
\end{equation}
for a  given $\phi_0\in(a,b)$.
Since $F(\phi)$ has the factor $(\phi-a)^2$,  $\lim_{\phi\to a^+}t(\phi)=-\infty$. In the following we will discuss more details of the solutions  (\ref{sol1}) for different signs of $c$ and finish the proofs of Theorems \ref{thm1} and \ref{thm2}.

\vspace{2mm}

 {\bf 1. Case  $c=0$.} In this case,
$$F(\phi)=\phi^n -na^{n-1}\phi+(n-1)a^n.$$
Since the only root of $F'(\phi)=0$ is $a$, $F(\phi)$ obtains its minimum at the point $a$  and $F(\phi)>0$ for all $\phi>a$.  As $\phi\to \infty$,  $\frac{\phi^{n-1}}{F(\phi)}\to \frac{1}{\phi}$ and  solution (\ref{sol1}) has the property that $t\simeq \log\phi$. Hence $\phi=\phi(t)$ is defined on the entire punctured  space $\mathbb C^{n\ast}$. Therefore there exist a family (depending on $a>0$) of zero cscK metrics with K\"ahler potential $u(t)$ such that $u'(t)=\phi(t)$.

Since $\phi\to a$ as $t\to -\infty $,
$$dt=\frac{\phi^{n-1}d\phi}{F(\phi)} \sim \frac{2a}{n(n-1)} \frac{1}{(\phi-a)^2}d\phi.$$
It turns out that from $\phi(t)=u'(t)=r^2u'(r^2)$,
$$ u(r^2)\sim a\log r^2-\frac{2a}{n(n-1)}\log(-\log r^2).$$
Moreover, by L'H\^osital's rule we get the more accurate expression of $u(r^2)$:
$$u(r^2)= a\log r^2-\frac{2a}{n(n-1)}\log(-\log r^2)+O((\log r^2)^{-1}).$$

On the other hand, we divide the case into $n=2$ and $n\geq 3$ to discuss the approximation of the solution as $r^2\to\infty$. For $n=2$,
$$dt=\frac{\phi d\phi}{(\phi-a)^2}=\bigl(\frac{1}{\phi-a}+\frac{a}{(\phi-a)^2}\bigr)d\phi$$
and it follows that

\begin{equation}\label{70901}
t=\log (\phi-a)+\frac 1 {(\phi-a)}
\end{equation}
 and $u(r^2)\simeq a\log r^2+ r^2$. Obviously, the derived metric is complete at entire $\mathbb C^{n*}$. Moreover, L'H\^osital's rule can be used to get more accurate estimate
$$u(r^2)=r^2+2a\log r^2+\frac{a^2}{2r^2}+{O}(\frac{1}{r^4}).$$

For $n\geq 3$, as $r^2\to \infty$
$$dt=\frac{1}{\phi}\bigl(1+na^{n-1}\phi^{1-n}-(n-1)a^n\phi^{-n}+ O(\phi^{-n-1})\bigr)$$
and then
\begin{equation*}\label{tor}
  t=\log \phi+\frac{n}{1-n}a^{n-1}\phi^{1-n}+\frac{n-1}{n}a^n\phi^{-n}+ O(\phi^{-n-1})
\end{equation*}
which implies
\begin{align*}
  \phi-e^t &=\phi\big(1-\exp(\frac{n}{1-n}a^{n-1}\phi^{1-n}+\frac{n-1}{n}a^n\phi^{-n}+ O(\phi^{-n-1})\bigr)\\
  &=\frac{n}{n-1}a^{n-1}\phi^{2-n}-\frac{n-1}{n}a^n\phi^{1-n}-\frac{n^2}{(1-n)^2}a^{2n-2}\phi^{2-2n}+ O(\phi^{-n-1}\bigr).
\end{align*}
Replacing $\phi$ by $e^t$ in the right hand side of the above equality, we have
$$\phi=e^t+\frac{n}{n-1}a^{n-1} (e^{t})^{2-n}-\frac{n-1}{n}a^n (e^t)^{1-n}+ O((e^t)^{-n-1})$$ or
$$
u(r^2)=r^2-\frac{na^{n-1}}{(n-1)(n-2)}(r^2)^{1-2n}+\frac{a^n}{n}(r^2)^{1-n}+{O}((r^2)^{-n}).
$$
Thus we have finished the proof of theorem \ref{thm1}.

We give the picture of $\phi=\phi(t)$ with $n=2$ and $a=1$ as Figure \ref{p1}. Recall that in this situation the function $\phi=\phi(t)$ is defined in equation (\ref{70901}).
Note that we also have
$$\det(g)=e^{-2t}\phi(t) \phi'(t)=\exp(-\frac{2}{\phi-1}).$$
 We give a rotational picture of the function $\exp(-\frac{2}{\phi-1})$ as Figure \ref{p2} which shows the  ALE and PMY properties of the metric.

\begin{figure}[H]
\centering
  \includegraphics[width=0.5\textwidth]{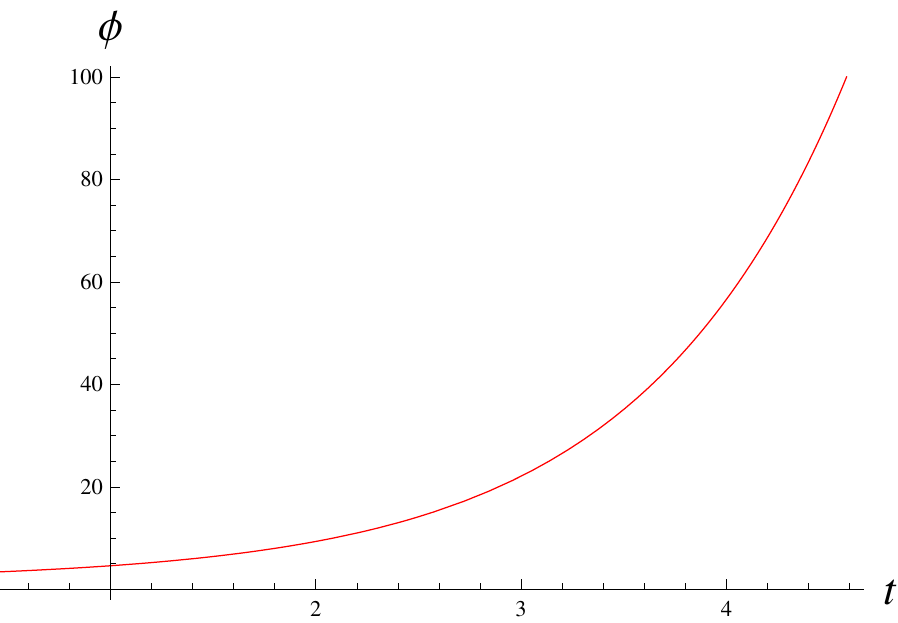}
  \caption{\small \color{blue}The graph of $\phi(t)$ with $c=0$ and $a=1$.}\label{p1}
\end{figure}

\begin{figure}[H]
\includegraphics[width=0.6\textwidth]{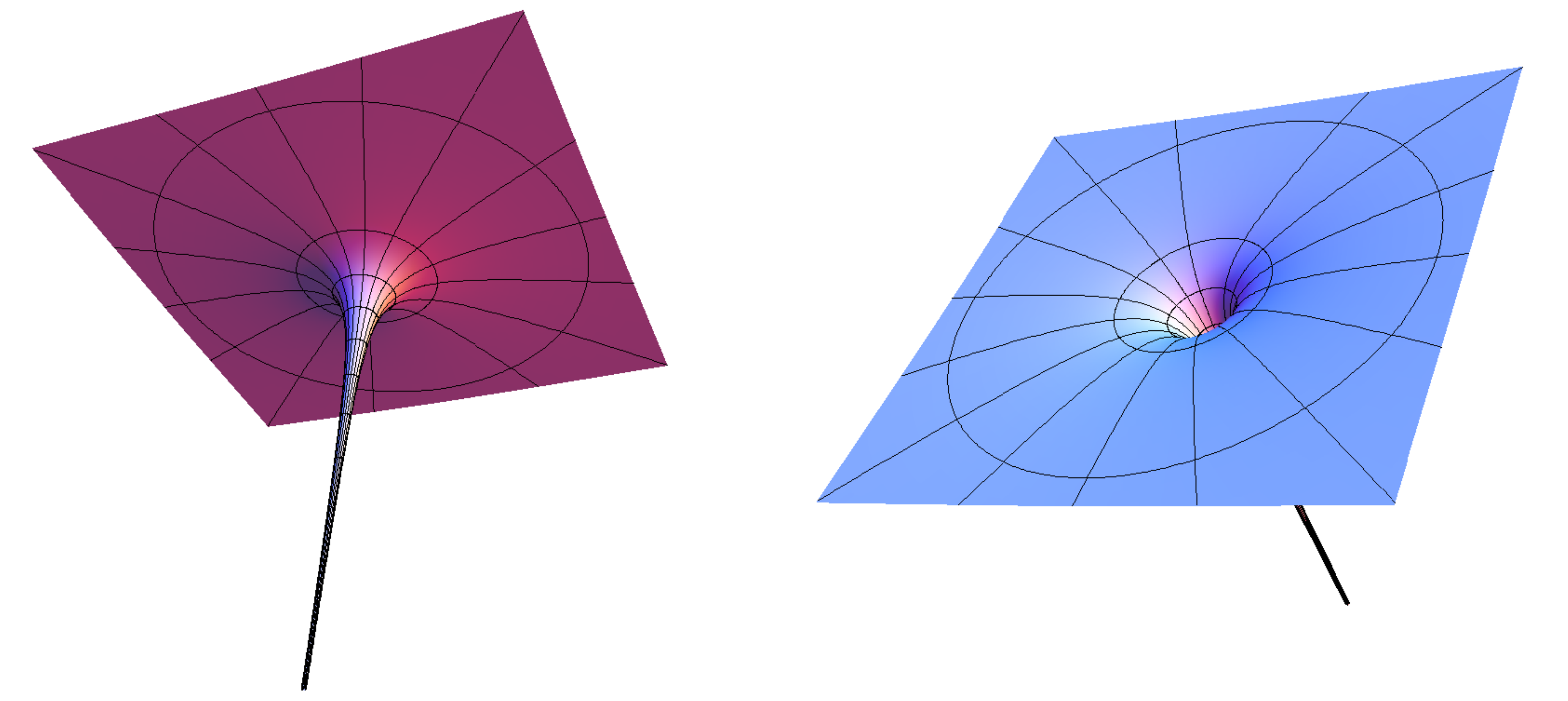}
\caption {\small \color{blue} PMY and ALE. }\label{p2}
\end{figure}
\vspace{2mm}

{\bf 2. Case  $c<0$.} In this case, we have seen that $F(\phi)>0$ when $\phi\in(a,+\infty)$.
From (\ref{sol1}), we also see that when $\phi\to+\infty$, the upper bound of $t=t(\phi)$ exists since the degree of $F$ is $n+1$. For simplicity, we take this upper bound to be zero since the solution (\ref{sol1}) is unique up to be a constant. Then $u(r^2)$ is defined on the punctured unit disc $D^{n\ast}$.

 The analysis of the boundary behavior is as follows.  As $\phi\to\infty$,
  $$t= \frac{n(n+1)}{-c}\frac{1}{\phi}+O(\frac{1}{\phi^2})$$
 which implies
 $$u(r^2)= \frac{n(n+1)}{-c}\log(-\log r^2)+O(\log r^2)$$ or $$u(r^2)= \frac{n(n+1)}{-c}\log (1-r^2)+O(\log r^2),$$
 where the right hand side  is the K\"ahler potential of the standard Poincar\'e metric on $D^n$. Hence, the metric we constructed is also complete near the boundary of $D^n$.

For the asymptotic behavior of $\phi=\phi(t)$ at the origin, it is the same as for the case $c=0$:  As $r^2\to 0$,
$$u(r^2)= a\log r^2-\frac{2a}{n(n-1)-ac}\log(-\log r^2)+{O}((\log r^2)^{-1}).$$
Then we have finished the proof of Theorem \ref{thm2} for the case $c<0$.

We give the picture of $\phi=\phi(t)$ as $n=2$, $a=1$ and $c=-6$ as Figure \ref{p3}.
Note in this situation,
 $$\frac{d\phi}{dt}=\frac{\phi}{(\phi-1)^2(\phi+3)}.$$

\begin{figure}[H]
\centering
  \includegraphics[width=0.5\textwidth]{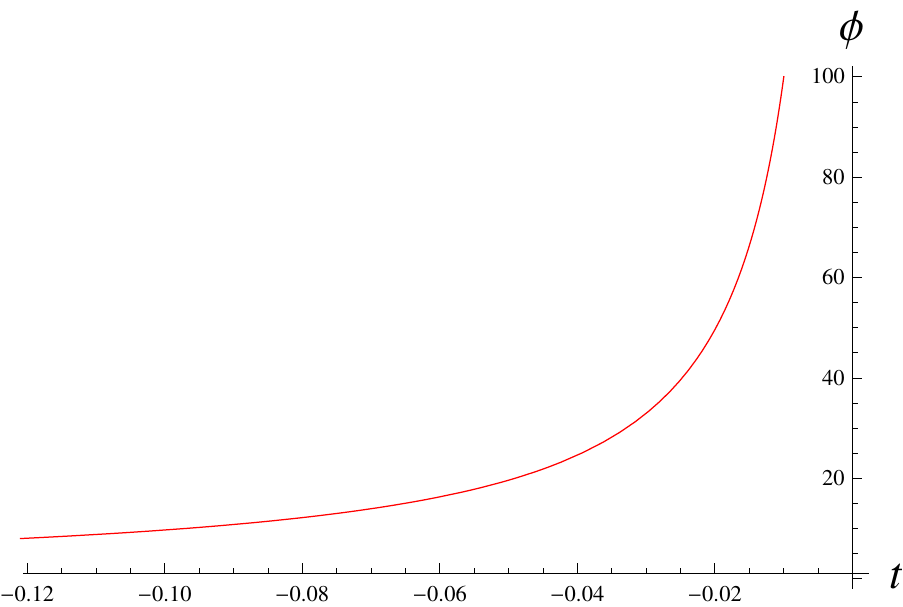}
  \caption{\small \color{blue} The graph of $\phi(t)$ with $c=-6$.}\label{p3}
\end{figure}
\vspace{2mm}

{\bf 3. Case $c>0$.} We have seen that if the constants $c$ and $a$ satisfy the relation $ac< n(n-1)$, then on domain $(a,\frac {n(n-1)}{c})$, $F(\phi)>0$. Obviously, when $\phi$ is big enough, $F(\phi)<0$. Hence we can let $b$ be the first number in $(a,+\infty)$ such that $F(b)=0$.
It follows that there is a polynomial $G(\phi)$ such that we can write
$$F(\phi)=\frac{c}{n(n+1)}(\phi-a)^2(b-\phi)G(\phi).$$
We first claim $G(b)>0$.  If $G(b)=0$, then $F(b)=F'(b)=0$. Together with $F(a)=F'(a)=0$, there are at least two different positive roots for the equation $F''(\phi)=0$. However, equation
$$F''(\phi)=-c\phi^{n-1}+n(n-1)\phi^{n-2} = 0$$
has only one positive root $\phi=\frac{n(n-1)}{c}$, which leads to a contradiction.
Hence, as $\phi\to b$,
\begin{equation}\label{ffd}
t\sim -\kappa\log (b-\phi),
\end{equation}
where
$$\kappa=-\frac{b^{n-1}}{F'(b)}>0.$$
This implies that when $\phi\in(a,b)$, $t\in (-\infty,+\infty)$ and then the K\"ahler potential $u(r^2)$ is defined on entire $\mathbb C^{n*}$.

The approximation of $u(r^2)$ near the zero is the same as for the case $c=0$:
 $$u(r^2)= a\log r^2-\frac{2a}{n(n-1)-ac}\log(-\log r^2)+{O}((\log r^2)^{-1}).$$
Whereas when  $t\to\infty$, from (\ref{ffd}) we can derive
$$u(r^2)= b\log r^2+\kappa r^{-\frac{2}{\kappa}}+O(r^{-2}).$$
The metric is not complete as $r^2\to \infty$. In fact as in the proof of Lemma \ref{lemC}, the length of the geodesic ray $\gamma(s)=sp$ on domain $(1,+\infty)$ is
\begin{equation*}
    l=\int_{\phi(\log(r^2(p)))}^{b}\sqrt{\frac{\phi^{n-1}}{F(\phi)}}d\phi  \simeq \int_{\phi(\log(r^2(p)))}^b \frac{1}{\kappa\sqrt{\phi-b}}d\phi <\infty.
  \end{equation*}
Thus we have finished the proof of Theorem \ref{thm2} for the case $c>0$.

We give the picture of function  $\phi=\phi(t)$ in case $n=2$, $a=1$ and $c=1$ as Figure \ref{p4}. In this situation,
 $$
\frac{d\phi}{dt}=\frac{6\phi}{(\phi-1)^2(4-\phi)}.
$$
\begin{figure}[hb]
\centering
  \includegraphics[width=0.5\textwidth]{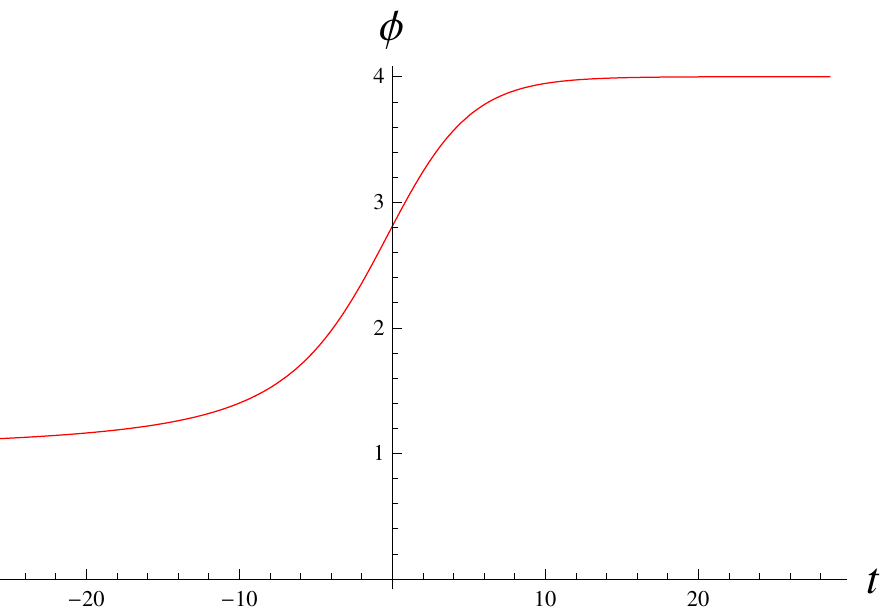}
  \caption{\small \color{blue} The graph of $\phi(t)$ with $c=1$.}\label{p4}
\end{figure}

\subsection{Further remarks}

1. For $n = 1$, $\partial\bar\partial \log r^2=0$ and
$$\sqrt{-1}\partial\bar\partial(-\log(-\log r^2))$$
is the standard Poincar\'e metric on $D^*$ with Gauss curvature $-1$. One can also construct on $\mathbb C^*$ a complete metric with zero Gauss curvature
$$\sqrt{-1}\partial\bar\partial(\log r^2)^2=\sqrt{-1} \frac{dz\wedge d\bar z}{r^2}.$$

2. It is mentioned in Introduction that the Mok--Yau metric defined on
$D^{2*}$ has good properties. One can see that $\sqrt{-1}\partial\bar\partial(-\log(-\log r^2))$
is also a K\"ahler metric on $D^{2*}$. However, its scalar curvature is infinity as $r^2\to 0$. In
fact, the term $\log r^2$ in the Mok--Yau metric results in the boundedness of the
scalar curvature near the punctured point.
Hence, the asymptotic property appeared  in Definition \ref{def} is
named as the PMY asymptotic property.
\vspace{2mm}

3. Write $\mathbb CP^n=\mathbb C^n\cup\mathbb CP^{n-1}$ and viewed  zero as a point $p$ in $\mathbb P^n$. One will ask whether the metric on $\mathbb C^{n\ast}$ constructed above with $c>0$ can be extended across $\mathbb CP^{n-1}$.  This is impossible. In fact,
it can be seen form  Lemma \ref{lem4} in section \ref{SPEC} below that the metric $\sqrt{-1}\partial\bar\partial u(t)$ can be extended across  $\mathbb CP^{n-1}$ if and only if
\begin{equation}\label{kap}
  \kappa=-\frac{b^{n-1}}{F'(b)}=1.
\end{equation}
Replacing  $\phi$ by $a\phi$ and $c$ by $a^{-1}c$, we can assume that $a=1$ in $F(\phi)$.   From $\kappa=-\frac{b^{n-1}}{F'(b)}$ and $F(b)=0$,
 we get the relation
$$b=\frac{n^2-1-c+(1-\frac{1}{\kappa})b^{n}}{n^2-c}.$$
Then $b>a=1$ implies $\kappa\not=1$,  which is a contradiction to  (\ref{kap}).
 So in this way we can not get a complete cscK metric on $\mathbb CP^{n}- p$.

In our another paper  \cite{FYZ2}, it has been proved that there also do not exist any complete cscK metrics on $\mathbb CP^n-p$ with PMY asymptotic property. Nevertheless,  a family of complete extremal metrics on $\mathbb CP^n-p$ have been constructed in \cite{FYZ2}.
\vspace{2mm}

4. At last, we give a simple proof of Theorem \ref{thm3} which states that there are not any complete K--E metrics on $\mathbb C^{n*}$ or on $D^{n*}$.

\begin{proof}[Proof of Theorem \ref{thm3}] Since $\mathbb C^{n\ast}$ and $D^{n\ast}$ are not compact, by Myers' theorem, we only need to consider the cases $c\leq 0$. In \cite{MY}, Mok and Yau proved that if a bounded domain $\Omega$ admits a complete hermitian metric such that $-C\leq \text{Ricci curvature}\leq 0$, then $\Omega $ is a domain of holomorphy. Since $D^{n*}$ is not a holomorphic domain, it does not admit a complete E--K metric with $c\leq 0$.

 For the nonexistence of negative K--E metrics on $\mathbb C^{n*}$, we use generalized Yau's Schwarz Lemma.
 \begin{lemma}\cite{MY,Yau2}
 Let $(M,\omega_g)$ be a complete hermitian manifold with scalar curvature bounded below by  $-K_1$ and let $(N,\omega_h)$ be a hermitian manifold of the same dimension  with Ricci curvature $\text{Ric}\leq -K_2\omega_h$ for some $K_2>0$.  If $f:M\to N$ is a holomorphic map and the Jacobian is nonvanishing at one point, then $K_1>0$ and
 $$f^*\omega_h^n\leq \left(\frac{K_1}{nK_2}\right)^n\omega_g^n.$$
 \end{lemma}
Now take $M=N=\mathbb C^{n\ast}$ and take the metric $\omega_g$ on $M$ as in Theorem \ref{thm1}. Then the above lemma leads to the nonexistence of negative K--E metric on $N=\mathbb C^{n\ast}$.

As for the Ricci--flat case, if we let $\omega=\sqrt{-1}g_{i\bar j}dz_i\wedge d\bar z_j$ be a complete Ricci--flat metric on $\mathbb C^{n\ast}$, then the function $\log\det(g_{i\bar j})$ is  pluriharmonic on it. Since the de Rham cohomology group $H^1_{dR}(\mathbb C^{n\ast},\mathbb R)$ vanishes,  $\log\det(g_{i\bar j})$ is the real part of a holomorphic function. By Hartogs' extension theorem for holomorphic functions, $\log\det(g_{i\bar j})$ is pluriharmonic on the entire space $\mathbb C^n$. Hence, $(g_{i\bar j})>C(\delta_{i\bar j})$ for some positive constant $C$.   Then following the discussions on page 49 of \cite{MY}, one can get a contradiction to the completeness of the metric near the origin of $\mathbb C^n$.
\end{proof}

\section{A momentum construction of complete cscK metrics}\label{sec3}

This section is devoted to prove Theorem \ref{thmlag0}. The first subsection almost follows the paper \cite{HS}. That is we first use the Calabi's
ansatz to derive an ODE and then use the moment profile to simplify it. In the second subsection completeness of metrics near zero section and at infinity are used to get constraint conditions. In the third subsection, we then consider the existence of metrics and their asymptotic property.


\subsection{The momentum construction}\label{momc}
Let $M$ be a compact K\"ahler manifold with a K\"ahler metric $\omega_M$. Let $\pi:L\to M$ be a holomorphic line bundle with a hermitian metric $h$. For any point $q\in M$, there is a holomorphic coordinate system $(U,z=(z_1,\cdots,z_m))$ of $q$ with $z(q)=0$ such that
$L|_{U}$ is holomorphically trivial. Under this trivialization, the hermitian metric $h$ can be given by a positive function $h(z)$.
Assume that $\omega_M$ and $h$ satisfy the condition
\begin{equation}\label{70101}
\sqrt{-1}\partial\bar\partial\log h=\lambda \omega_M,\qquad \text{for some constant $\lambda$}.
\end{equation}

Let $E$ be the direct sum of $n\ (\geq 2)$ copies of $L$, i.e. $E=L^{\oplus n}$,
with an associated hermitian metric still denoted by $h$. We also denote $\pi:E\to M$. We have a local trivialization of $E$ induced
from one of $L$ and denote the fiber coordinates by $w=(w_1,\cdots,w_n)$.  In the following, we denote $\alpha$ and $\beta$ as the lower index of the components of $w$ and $i$ and $j$ as the lower index of the components of $z$.
There is a fibrewise norm squared function $r^2$ on the total space of $E$ defined by $h$
$$r^2=h(z)\sum_{\alpha=1}^n|w_\alpha|^2.$$
If $M$ is viewed as the zero section of $E$, i.e. the set defined by $r^2=0$, we want to construct complete cscK metrics on $E- M$ under condition (\ref{70101}). In this section denote $E-M$ simply by $E^\ast$.

Let
\begin{equation}\label{70405}
\nu=\log r^2=\log h(z)+t,\qquad \text{with}\qquad t=\log\bigl(\sum_{\alpha=1}^n|w_\alpha|^2\bigr).
\end{equation}
Consider Calabi's ansatz
\begin{equation*}
  \omega=\pi^*\omega_M + \sqrt{-1}\partial\bar\partial f(\nu).
\end{equation*}
Using condition (\ref{70101}), we have
\begin{equation*}
\begin{aligned}
\omega=&(1+\lambda f'(v))\pi^\ast\omega_M+f''(\nu)\sqrt{-1}\partial\log h\wedge \bar\partial \log h \\
&+f'(v)\cdot\sqrt{-1}(\partial \log h(z)\wedge \bar\partial t+\partial t\wedge\bar\partial \log h(z))\\
&+f'(v)\cdot \sqrt{-1}\partial\bar\partial t+f''(v)\cdot \sqrt{-1}\partial t\wedge \bar\partial t
\end{aligned}
\end{equation*}
and
\begin{equation}\label{70102}
\omega^{m+n}=(1+\lambda f'(v))^m\pi^\ast\omega_M^m\wedge (f'(v)\cdot \sqrt{-1}\partial\bar\partial t+f''(v)\cdot \sqrt{-1}\partial t\wedge \bar\partial t)^n.
\end{equation}
The reason one can derived the above equality is $\det (\frac{\partial^2 t}{\partial w_\alpha\partial\bar w_\beta})=0$.  In practise,  when computing at any point
$p\in \pi^{-1}(q)$, one can let $w_\alpha(p)=0$ for $2\leq \alpha\leq \beta$.
Then $\omega$ is (strictly) positive  if and only if
\begin{equation}\label{pos2}
f'(\nu)>0,\qquad f''(\nu)>0,\qquad \text{and}\qquad 1+\lambda f'(\nu)>0.
\end{equation}

\begin{definition}
The above constructed metric $\omega$ is call a { bundle adapted metric}.
\end{definition}

Next we compute the Ricci curvature and the scalar curvature of the bundle adapted metric $\omega$.
From (\ref{70102}), we have
\begin{equation}\label{det}
\det(\omega)=\det(\omega_M)\cdot (1+\lambda f'(\nu))^me^{-nt}f'(\nu)^{n-1}f''(\nu).
\end{equation}
Let
$$\Psi(\nu)=\log \bigl((1+\lambda f'(\nu))^m f'(\nu)^{n-1} f''(\nu)\bigr).$$
For any $q\in M$, we can assume that the local coordinates $z=(z_1,\cdots,z_n)$ around $q$  also satisfy $\partial h|_{q}= \bar\partial h|_{q}=0$. Then under  assumption (\ref{70101}),
the Ricci form of $\omega$ at a point $p\in\pi^{-1}(q)$ is
\begin{align*}
  Ric(\omega)|_p=&-\sqrt{-1}\partial\bar\partial (\log\det(g_M)+\Psi(\nu)-nt)|_{p}\\
  =&Ric(\omega_M)|_q-\lambda \Psi'(\nu)\omega_M|_q\\
  &-(\Psi'(\nu)-n)\sqrt{-1}\partial\bar\partial t|_p-\Psi''(\nu)\sqrt{-1}\partial t\wedge \bar\partial t|_p,
\end{align*}
where $Ric(\omega_M)$ is the Ricci form of $\omega_M$ on $M$. The matrix composed by components of metric $\omega$ at $p$ is
\begin{equation*}
\left (\begin{array}{cc}
(1+\lambda f'(\nu))(g_{i\bar j})_{m\times m} & 0\\
0 & \bigl(f'(\nu)\delta_{\alpha\beta}+f''(\nu)\bar w_\alpha w_\beta\bigr)_{n\times n}
\end{array}
\right),
\end{equation*}
where $(g_{i\bar j})_{m\times m}$ is the coefficients matrix of metric $\omega_M$.
Its inverse matrix is
$$\begin{pmatrix}
\frac{1}{1+ \lambda f'(\nu)}(g^{i\bar j})_{n\times n} & 0 \\ 0 & (\frac{e^t}{f'(\nu)}\delta_{\alpha\beta}+w_\alpha\bar w_\beta(\frac{1}{f''(\nu)}-\frac{1}{f'(\nu)}))_{n\times n}\end{pmatrix}.$$
If $c_M$ denotes the scalar curvature of $\omega_M$, the scalar curvature of $\omega$ at the point $p$ is
\begin{equation}\label{cnu}
  c=\frac{c_M}{1+ \lambda f'(\nu)}-\frac{\lambda m\Psi'(\nu)}{1+\lambda f'(\nu)}-(n-1)\frac{\Psi'(\nu)-n}{f'(\nu)}-\frac{\Psi''(\nu)}{f''(\nu)}.
\end{equation}
The above formula of scalar curvature is obviously globally defined.

Usually it is more suitable to use the  Legendre transform to solve scalar curvature equation (\ref{cnu}). From the positivity (\ref{pos2}) of $\omega$, $f(\nu)$ must be strictly convex. Then one can take the Legendre transform $\mathcal{F}(\tau)$ of $f(\nu)$. The Legendre transform $\mathcal{F}(\tau)$ is defined in term of the variable $\tau=f'(\nu)$ by the formula
$$f(\nu)+\mathcal{F}(\tau)=\nu\tau.$$
Let $I\subset \mathbb{R}_+$ be the image of $f'(\nu)$. The \emph{momentum profile} $\varphi(\tau)$ of the metric is defined to be $\varphi:I\to \mathbb{R}$,
$$\varphi=\frac{1}{\mathcal{F}''(\tau)}.$$
Then the following relations can be verified:
$$\varphi(\tau)=f''(\nu)\qquad  \text{and}\qquad  \frac{d\tau}{d\nu}=\varphi. $$
Also, $\nu$ can be viewed as a function of $\tau$, up to a constant,
\begin{equation}\label{ttau}
  \nu(\tau)=\int \frac{1}{\varphi(\tau)} d\tau.
\end{equation}
The advantage of the Legendre transform is that the scalar curvature of $\omega$ can be described in terms of $\varphi(\tau)$ and the domain $I$ of $\tau$. Especially, the boundary completeness (see the next subsection) and  the extendability properties (see the next section) can also be read off from the behaviour of $\phi(\tau)$  near the end points of $I$.

Using these transformations, we have
$$\Psi(\tau)=\log((1+\lambda \tau)^m\tau^{n-1}\varphi(\tau)).$$
Let
$$Q(\tau)=(1+\lambda \tau)^m \tau^{n-1}.$$
Then $\Psi=\log (Q\varphi)$.
By direct computation,
\begin{equation*}
\begin{aligned}
&\Psi'(\nu)=\frac{1}{Q}\frac{\partial(Q\phi)}{\partial\tau}\\
&\Psi''(\nu)=-(m\lambda\tau^{n-1}+(n-1)(1+\lambda\tau))\frac{\varphi}{Q^2}\frac{\partial(Q\varphi)}{\partial\tau}+\frac{\varphi}{Q}\frac{\partial^2(Q\varphi)}{\partial \tau^2}.
\end{aligned}
\end{equation*}
Inserting the above equalities into (\ref{cnu}), and replacing $f'(\nu)$ by $\tau$, then simplifying it, at last we obtain
\begin{equation}\label{70201}
c= \frac{c_M}{1+ \lambda \tau}+\frac{n(n-1)}{\tau}-\frac{1}{Q}\frac{\partial^2 (Q\varphi)}{\partial^2\tau}.
\end{equation}

Now we assume that $c_M$ is a constant. We want to look for $\varphi$ such that $c$ is a constant. By integrations, we  solve  equation (\ref{70201}) as
\begin{equation}\label{70401}
\begin{aligned}
  &(\varphi Q)(\tau)=(\varphi Q)(a)+(\varphi Q)'(a)(\tau-a)+P(\tau),\qquad \text{with}\\
  &P(\tau)=\int_{a}^\tau (\tau-x)\biggl(\frac{c_M}{1+ \lambda x}+\frac{n(n-1)}{x}-c\biggr)Q(x)dx,
  \end{aligned}
\end{equation}
where $a$ is the left endpoint of $I$. Here $(\varphi Q)(a)$ and $(\varphi Q)'(a)$ are constants. $P(\tau)$ is a polynomial and hence  $\varphi(\tau)$ is  a real rational function.

\subsection{Completeness}
Assume that $I=(a,b)\subset \mathbb R^{+}$, $b$ may be infinity,  is the maximum interval where $\varphi(\tau)$ determined by (\ref{70401}) is defined and positive. Further assume that
\begin{equation*}
\lim_{\tau\to a^{+}}\nu(\tau)=-\infty
\end{equation*}
 and
 \begin{equation*}
 1+\lambda \tau>0,\qquad \text{when}\ \tau\in(a,b).
 \end{equation*}
 Then from the above subsection, the bundle adapt metric $\omega$ is well-defined on $\mathbb U^\ast(\nu(b))=\{p\in E\ |\ -\infty<\nu(p)<\nu(b)\}\subset E^\ast$. Here  $\nu(b)=\lim_{\tau\to b^{-}}\nu(\tau)$. If $\nu(b)$ is infinity, then $\mathbb U^\ast(\nu(b))=E^\ast$, and if $\nu(b)$ is a constant, we can take an integration constant in (\ref{ttau}) such that $\nu(b)=0$ and hence $\mathbb U^\ast(\nu(b))$ is $\mathbb U^\ast=\mathbb  U-M$ as defined in Introduction. We first establish the following lemma.

\begin{lemma}\label{070501}
Under the above assumptions, the bundle adapt metric $\omega$ is complete near the zero (punctured) section if and only if $a>0$ and $\varphi(a)=\varphi'(a)=0$. Thus, $\varphi(\tau)=\frac{P(\tau)}{Q(\tau)}$.

Moreover, if $b$ is finite, then $\omega$ is defined on $E^\ast$ and $\omega$ is complete if and only if $\varphi$ also satisfies $\varphi(b)=\varphi'(b)=0$; Whereas if $b$ is infinity, then $\omega$ is defined on $E^\ast$ or on $\mathbb U^\ast$ and is automatically complete.
\end{lemma}

\begin{proof}
For any point $q\in M$, $p\in \pi^{-1}(q)\cap \mathbb U^\ast(\nu(b))$, we consider the ray starting  from $q$ on the fiber $\pi^{-1}(q)$:
\begin{equation}\label{70406}
\gamma(s)=s\cdot p,\qquad s\in (0,1]\ \  \text{or}\ \ s\in [1,s_0)
\end{equation}
for $s_0^2=\exp(\nu(b)-\nu(p))$. Such an $s_0$ can be derived from  the following calculation by (\ref{70405}):
\begin{equation*}
\nu(b)=\nu(s_0\cdot p)=\log h(z(q))+t(s_0\cdot p)=\log s_0^2+\nu(p).
\end{equation*}

Since $\pi^{-1}(q)\cap \mathbb U^\ast(\nu(b))$ is a totally geodesic submanifold of $(\mathbb U^\ast(\nu(b)),\omega)$ and the induced metric on $\pi^{-1}(q)\cap\mathbb U^\ast(\nu(b))$ is $U(n)$-invariant, $\gamma(s)$ is a geodesic on $\pi^{-1}(q)\cap\mathbb U^\ast(\nu(b))$, and hence is also a geodesic on $\mathbb U^\ast(\nu(b))$. Also since $M$ is compact, the metric $\omega$ is complete if and only of the lengths of the rays $\gamma(s)$ defined in (\ref{70406}) are infinity.

As done in Lemma \ref{lemC}, the length of $\gamma(s)$ on domain  $(0,1]$ is
\begin{equation*}
l_1=\int_0^1|\gamma'(s)|ds=\frac 1 2 \int_{-\infty}^{\nu(p)}\sqrt{f''(\nu)}d\nu=\frac 12 \int_a^{\tau(\nu(p))}\frac{1}{\sqrt{\varphi(\tau)}}d\tau.
\end{equation*}
The completeness near the zero section requires $l_1=+\infty$, which is equivalent to that $\varphi(\tau)$ has a factor $(\tau-a)^2$, i.e.  $\varphi(a)=\varphi'(a)=0$.
We claim $a>0$. If $a=0$,  the lowest degree term of polynomial $P(\tau)$ defined in (\ref{70401}) would  be determined as
\begin{equation*}
\int_0^\tau(\tau-x)n(n-1)x^{n-2}dx=\tau^n.
\end{equation*}
Hence we can write $P(\tau)$ as $P(\tau)=\tau^n(1+A(\tau))$ for some polynomial $A(\tau)$ and thus get
\begin{equation*}
\varphi(\tau)=\frac{P(\tau)}{Q(\tau)}=\tau\frac{1+A(\tau)}{(1+\lambda\tau)^m}.
\end{equation*}
In this way we find $l_1<+\infty$, which is a contradiction.

Next, we should consider the endpoint $b$. The length $l_2$ of $\gamma(s)$ for $s\in [1,s_0)$  is
\begin{equation*}
l_2=\int_1^{s_0}|\gamma'(s)|ds=\frac 1 2\int_{\nu(p)}^{\nu(s_0\cdot p)=\nu(b)}\sqrt{f''(\nu)}d\nu=\frac 1 2 \int_{\tau(\nu(p))}^{\tau(\nu((b)))=b}\frac 1{\sqrt{\varphi}}d\tau
\end{equation*}
If $b$ is finite, then $P(\tau)$ has a factor $(\tau-b)$. By (\ref{ttau}), $\lim_{\tau\to b^{-}}\nu(\tau)=+\infty$. Hence, $\omega$ is well-defined on $E^\ast$. If $\omega$ is complete, the above $l_2$ is also infinity, which is equivalent to say that $\varphi(b)=\varphi'(b)=0$. If $b$ is infinity, then $\varphi(\tau)$ is defined on $(a,+\infty)$.  Since  $\wp\triangleq\deg P(\tau)-\deg Q(\tau)=2$ or 1, from (\ref{ttau}) if $\wp=2$, $\omega$ is defined on $\mathbb U^\ast$ and if $\wp=1$, $\omega$ is defined on $E^\ast$.   Also since  $\wp=2$ or $1$,  there exists a constant $C$ big enough  such that $\varphi(\tau)<C\tau$ as $\tau\to\infty$. Hence, in this situation, $l_2$ is infinity and $\omega$ is automatically complete.
\end{proof}

\subsection{Existence of complete cscK metrics}\label{pfthm4}
We discuss the solutions in this subsection divided into three cases: $\lambda> 0$, $\lambda=0$ and $\lambda<0$.
\vspace{2mm}

 {\bf 1. Case $\lambda> 0$.} Given constants $c_M$, $\lambda> 0$, and $a>0$, define the set $\mathfrak C$ to be of ``allowable scalar curvatures" as
\begin{equation*}
\mathfrak C=\{\ c\in\mathbb R\ |\ \varphi(\tau)>0\ \ \textup{for}\ \ \tau\in (a,+\infty)\ \}.
\end{equation*}
$\mathfrak {C}$ is not empty since $\varphi(\tau)$ is positive if  $c<<0$.  $\mathcal{C}$ has a supermum. In fact, if $c>0$, $P(\tau)$ and hence $\varphi(\tau)$ will be negative when $\tau$ is big enough. Hence, the supermum $c_0$ of $\mathfrak C$ is nonpositive. We can easily get the conclusions:   If $c_M\geq 0$,  $c_0=0\in\mathfrak C$; If $c_M<0$,  two possibilities occur: one is $c_0\in\mathfrak C$, and the other is $c_0\notin \mathfrak C$, which means that there exists a constant $b$ such that $\varphi(b)=0$ and $\varphi(\tau)$ is positive on $(a,b)$.
Hence we should consider the existence of K\"ahler metrics with constant scalar curvature $c$ as the following four cases:
$$ (i)\ c<c_0; \ \  (ii)\ c=c_0=0\in\mathfrak C;\ \  (iii)\ 0> c=c_0\in\mathfrak C;\ \ \textup{or}\ (iv)\ c=c_0\not\in\mathfrak C.$$

\begin{proposition}\label{prop3}
Given constants $c_M$, $\lambda> 0$ and $a>0$, there exists a constant $c_0\leq 0$ such that:

1. For any $c\leq c_0$ in cases (i) and  (iii), there exists a complete cscK metric $\omega$ on $\mathbb U^\ast$ with constant scalar curvature $c$; and

2. For $c=c_0$ in cases (ii) and (iv), there exists a complete cscK metric $\omega$ on $E^\ast$ with constant scalar curvature $c$.
\end{proposition}
\begin{proof}
For cases (i) and (iii), since the degrees of polynomials $P(\tau)$ and $Q(\tau)$ are  $m+n+1$ and $m+n-1$ respectively, the limit of $\nu(\tau)$ defined by (\ref{ttau})  is finite as $\tau\to +\infty$. Set this constant to be zero. Then the metric $\omega$ is defined on $\mathbb U^\ast$. According to  Lemma \ref{070501}, $\omega$ is complete.

For case (ii),  $\deg P(\tau)-\deg Q(\tau)=1$. The limit of $\nu(\tau)$ is infinity as $\tau\to +\infty$. Hence the metric is defined on $E^\ast$ and is complete by Lemma \ref{070501}.

For case (iv),  $\varphi(\tau)\geq 0$ for $\tau\in (a,+\infty)$. In this case there exists a constant $b$ such that $\varphi(\tau)>0$ in $\tau\in (a,b)$ and $\varphi(b)=0$. Hence $\varphi'(b)=0$. According to Lemma \ref{070501}, $\omega$ is defined on $E^\ast$ and is complete.
\end{proof}


We give two examples.
\begin{example}\label{ee}
Consider the case $c_0\in\mathfrak C$ and $c_0<0$.
 Take
 $$c=\frac{c_M}{1+\lambda a}+\frac{n(n-1)}{a}.$$
We have
\begin{equation*}
\begin{aligned}
&\frac{c_M}{1+\lambda x}+\frac{n(n-1)}{x}-c\\
=&\frac{a-x}{a(1+\lambda a)x(1+\lambda x)}\left(\lambda(c_M a+n(n-1)(1+\lambda a))x+n(n-1)(1+\lambda a)\right).
\end{aligned}
\end{equation*}
If $c_M<0$, $\lambda>0$ and $a>0$ satisfy
\begin{equation*}
c_M=-\frac{n(n-1)(1+\lambda a)^2}{\lambda a^2},
\end{equation*}
we find that $P(\tau)>0$ when $\tau\in (a,+\infty)$. It is easy to check that
$$ c_0=c=-\frac{n(n-1)}{\lambda a^2}.$$
\end{example}

\begin{example}\label{bb=0} Then consider the case $c_0\notin \mathfrak C$. Let $m=1$, $n=2$, $\lambda=1$, $c_M=-4$  and $a=1$.  It follows that
  $$\varphi(\tau)=\int_1^\tau(\tau-x)\big(-cx^2-(c+2)x+2\big) dx.$$
  We can solve the inequality $\varphi(\tau)\geq 0$ to get $c\leq \psi(\tau)$. Here
  $$\psi(\tau)=\frac{\frac{1}{3} - \tau + \tau^2 - \frac{\tau^3}{3}}{\frac{7}{12}-
  \frac{5}{6} \tau + \frac{1}{6}\tau^3 + \frac{1}{12}\tau^4}.$$
Hence $\varphi(b)=\varphi'(b)=0$ if and only if $c_0=\min _{\tau\in (a,+\infty)}\psi(\tau)=\psi(b)$.

The pictures of $\psi(\tau)$ and $\varphi(\tau)$ are given as Figure \ref{c}. We find that $\psi(\tau)$ achieves its maximum at $\tau=4.4641$ with the maximum $-0.3094$. So $c_0=-0.3094$ and $b=4.4641$. Thus, $\varphi(\tau)$ gives a complete cscK metric on $E^\ast$ with scalar curvature $c_0$.
   \begin{figure}[ht] \centering
   	\subfigure[ \color{blue}the graph of $\psi(\tau)$] { \label{a}
   		\includegraphics[width=0.5\columnwidth]{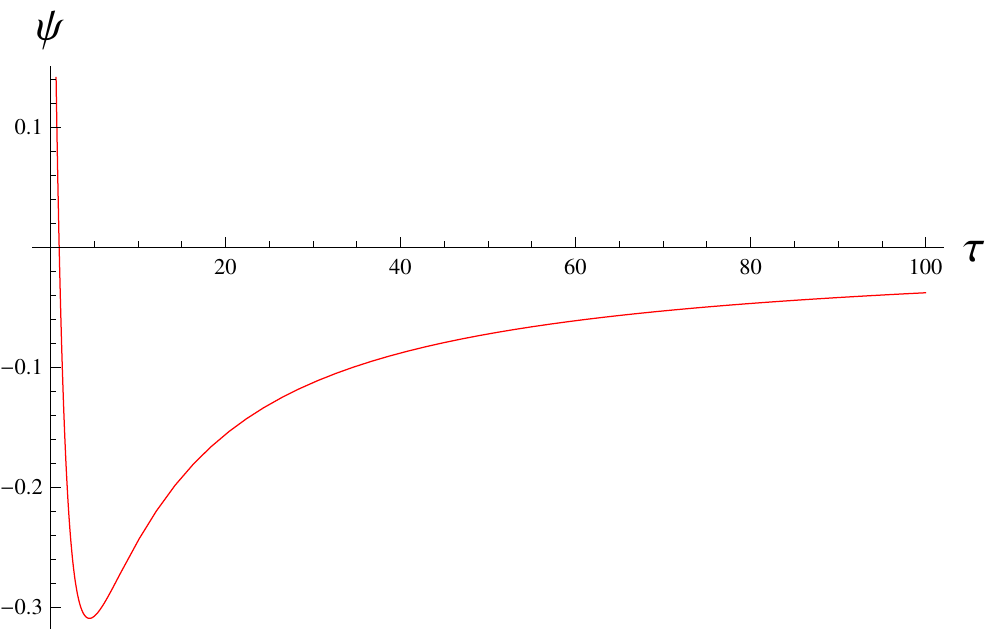}
   	}
   	\subfigure[\color{blue} the graph of $\varphi(\tau)$] { \label{b}
   		\includegraphics[width=0.5\columnwidth]{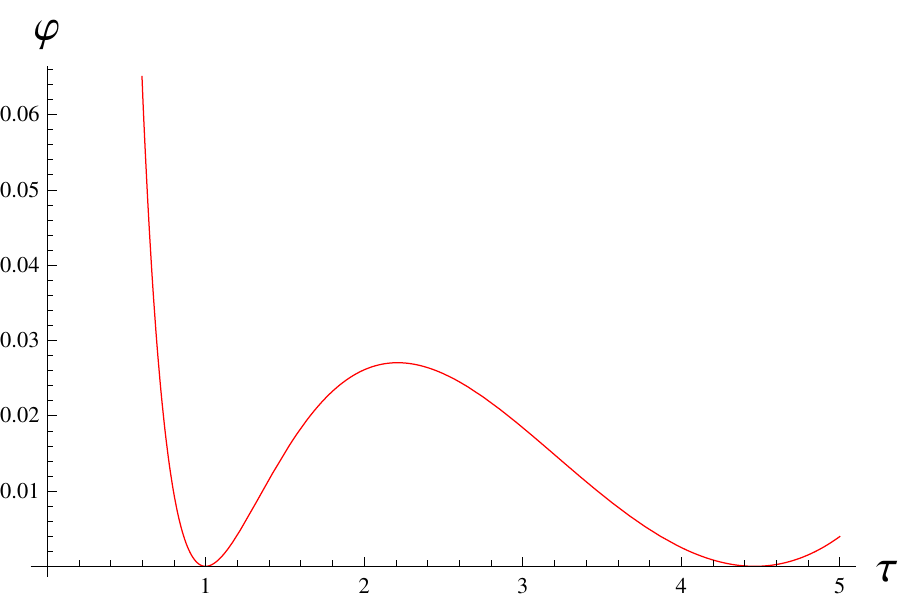}
   	}
   \caption{\small \color{blue}The case of $c_0\notin\mathfrak C$.}
   	\label{c}
   \end{figure}

\end{example}
 We consider the asymptotic property. Let
$$\kappa(\tau)=\frac{c_M}{1+\lambda\tau}+\frac{n(n-1)}{\tau}-c.$$
Since $\varphi(a)=\varphi'(a)=0$, $\kappa(a)=\varphi''(a)$.

\begin{proposition}\label{pop}
For cases (i), (ii) and (iv) the cscK metrics in Proposition \ref{prop3} have the PMY asymptotic property, and for case (iii), the metrics have the asymptotic property: As $r^2\to 0$,
$$ f(r^2)= a\log r^2-2(\frac 3{\kappa'(a)})^{\frac 1 2}(-\log r^2)^{\frac 1 2}+O(\log(-\log r^2)),$$
or
$$f(r^2)= a\log r^2-\frac 3 2 (\frac{8}{\kappa''(a)})^{\frac 1 3}(-\log r^2)^{\frac 2 3}+O((-\log r^2)^{\frac 1 2}).$$
\end{proposition}
\begin{proof}
For cases (i), (ii) and (iv), we claim that $\kappa(a)>0$.
If the claim holds, then as $r^2\to 0$,
$$\frac{d\tau}{d\nu}=\varphi(\tau)= \frac{\kappa(a)}{2}(\tau-a)^2+O(\frac{1}{\tau-a}),$$
 from which we can get
$$f(r^2)= a\log r^2-\frac{2}{\kappa(a)}\log(-\log r^2)+O((\log r^2)^{-1}),$$
which means that the metric is with PMY asymptotic property.

We prove the claim. For case (i), since $\frac{c_M}{1+\lambda a}+\frac{n(n-1)}{a}-c_0\geq 0$, $\kappa(a)=(\kappa(a)+c-c_0)+(c_0-c)>0$.
 For case (ii), if $\kappa(a)=0$, then $\varphi^{(3)}(a)=\kappa'(a)$ and $\varphi(\tau)$ has the Taylor expansion at $\tau=a$:
$$\varphi(\tau)=\frac{\kappa'(a)}{3!}(\tau-a)^3+o((\tau-a)^3).$$
The positivity of $\varphi(\tau)$ when $\tau>0$ implies $\kappa'(a)\geq 0$. On the other hand, if $\kappa(a)=0$, i.e., $\frac{c_M}{1+\lambda a}+\frac{n(n-1)}{a}=0$ as $c=0$, then
$$\kappa'(a)=\frac{-c_M}{(1+\lambda a)^2}-\frac{n(n-1)}{a^2}=-\frac{n(n-1)}{(1+\lambda a)a^2}<0,$$
which is a contradiction to $\kappa'(a)\geq 0 $. Hence $\kappa(a)\not=0$ and  $\kappa(a)>0$ is deduced from the positivity of $\varphi(\tau)$.
  We then consider case (iv). In this case, $I=(a,b)$ and
  $$\varphi(a)=\varphi'(a)=\varphi(b)=\varphi'(b)=0.$$
These equalities guarantee that there are  	already two roots in (a,b) for $\varphi''(\tau)=\kappa(\tau)Q(\tau)=0$. Hence $\kappa(a)=\varphi''(a)\not=0$. The positivity of $\varphi$ then implies $\kappa(a)>0$.

For case (iii), we first prove that $\kappa(a)=0$ which means that the metrics for this case is not PMY. Since $c=c_0<0$, there exist  constants $a_0\in(a,\infty)$ and  $C>0$ such that $\kappa(\tau)\geq C$ in $[a_0,\infty)$. If $\kappa(a)>0$,  there would exist constants $a_1>a$ and $C_1>0$ such that $\kappa(\tau)>C_1$ in $(a,a_1)$ and $\varphi(\tau)>C_1$ in $(a_1,a_0]$. Then we could choose a positive constant $\epsilon$ such that $\varphi(\tau)$ is still positive by replacing  $c=c_0$ with $ c=c_0+\epsilon$. Hence $c=c_0+\epsilon\in\mathfrak C$ which contradicts to that $c_0$ is the supermum of $\mathfrak C$.

In this case, there are two subcases which should be considered: $\kappa'(a)>0$ and $\kappa'(a)=0$: If $\kappa'(a)>0$,
$$f''(\nu)=\frac{d\tau}{d\nu}=\varphi(\tau)=\frac{\kappa'(a)}{3!}(\tau-a)^3+O((\tau-a)^4);$$
If $\kappa'(a)=0$, $\kappa''(a)\not=0$ and
$$f''(\nu)=\frac{d\tau}{d\nu}=\varphi(\tau)= \frac{\kappa''(a)}{4!}(\tau-a)^4+O((\tau-a)^5).$$
The conclusion then follows.
\end{proof}

\begin{proposition}\label{poop}
The metrics in Proposition \ref{prop3} with constant scalar curvature $c$ have the asymptotic property:

1. For cases (i) and (iii) (hence defined on $\mathbb U^\ast$), as $r^2\to 1$,  	
$$f(r^2)=-\frac{(m+n)(m+n+1)}{c} \log(-\log r^2)+O(\log r^2);$$

2. For case (ii) (hence defined on $ E^\ast$), as $r^2\to\infty$,
\begin{equation*}
		f(r^2)= \frac{1}{\theta_1} (r^2)^{\theta_1}+\theta_2 \log r^2+O(r^{-2})
		\end{equation*}
with $ \theta_1= \frac{c_M+n(n-1)\lambda}{\lambda(m+n)(m+n+1)}$ and $\theta_2=\frac{(m+n)(c_M (m-1)+n(n-1)m\lambda)}{m(m+n-2)(c_M+n(n-1)\lambda)}$;

3. For case (iv) (hence defined on $E^\ast$), as $r^2\to \infty$,
\begin{equation*}\label{b*b}
  f(r^2)= b\log r^2-\frac{2}{\kappa(b)}\log(\log r^2)+O((\log r^2)^{-1})\qquad  \text{with}\ \  \kappa(b)>0.
\end{equation*}
\end{proposition}
\begin{proof}
We omit the proof here. It is a calculus exercise.
\end{proof}

{\bf 2. Case $\lambda=0$.}
In this case,
$\kappa(\tau)=c_M+\frac{n(n-1)}{\tau}-c$. Hence $c_0=c_M$.
\begin{proposition}\label{prpmy}
Given constants $c_M$, $\lambda=0$, and $a>0$, there exists a complete cscK metrics on $\mathbb U^\ast$ with $c<c_M$ and on $E^\ast$ with $c=c_M$. All these metrics admit PMY asymptotic property.
\end{proposition}
\begin{proof}
The proof is the same as Propositions \ref{prop3} and \ref{pop}.
\end{proof}

 {\bf 3.  Case $\lambda<0$.}
The method of this case is quite different from the cases $\lambda\geq0$.
\begin{proposition}\label{ppppp}
For any $\lambda<0$ and  $c_M>0$, there exists on $E^\ast$ a complete positive cscK metric with PMY asymptotic property.
\end{proposition}

\begin{proof}
We need to prove that there exist  two  constants $a$ and $b$ with  $0<a<b<-\frac 1 \lambda$ such that the function $\varphi(\tau)$ is positive on domain  $(a,b)$ and $\varphi(b)=\varphi'(b)=0$.

 On interval $(0,-\frac 1 \lambda)$, the polynomial $Q(\tau)$ is positive. Hence there is a number $b\in(a,-\frac 1 \lambda)$ such that $\varphi(b)=\varphi'(b)=0$ if and only if $P(b)=P'(b)=0$, and when $\tau\in (a,b)$, $\varphi(\tau)$ is positive if and only if $P(\tau)$ is positive.

From $P(b)=P'(b)=0$, we can solve $c_M$ and $c$ as
\begin{equation}\label{cc}
c_M=n(n-1)\frac{H_1(a,b)}{H_2(a,b)}\qquad \textup{and}\qquad c=n(n-1)\frac{H_3(a,b)}{H_2(a,b)},
\end{equation}
where we have defined
\begin{equation}\label{HHH}
\begin{aligned}
&H_1(a,b)=\int^b_a\frac{Q(x)}{x}dx\int_a^bxQ(x)dx-\Bigl(\int_a^bQ(x)dx\Bigr)^2,\\
&H_2(a,b)=-\frac 1 \lambda \int^b_a\frac{Q(x)}{1+\lambda x}dx\int_a^b (1+\lambda x)Q(x)dx+\frac 1 \lambda \Bigl(\int_a^bQ(x)dx\Bigr)^2,\\
&H_3(a,b)=\int^b_a\frac{xQ(x)}{1+\lambda x}dx\int_a^b\frac{Q(x)}{x}dx-\int_a^bQ(x)dx\int_{a}^b\frac{Q(x)}{1+\lambda x}dx.
\end{aligned}
\end{equation}
We first note that when $\lambda<0$ and $0<a<b<-\frac 1 \lambda$, the functions $H_i(a,b)$ for $i=1,2,3$ are always positive. The proofs for the first and second functions are direct by the H\"older inequality. The proof for the third one is also direct by using the common techniques in calculus.
Thus the constant $c_M$ and $c$ defined in (\ref{cc}) are  indeed positive.

We need the following.
\vspace{1mm}

\noindent {\bf Claim:} For any given positive $c_M$, there exist two constants $a$ and $b$ with $0<a<b<-\frac 1 \lambda$ such that the first equality in (\ref{cc}) holds.
\begin{proof}
Define a function
$$ H(\zeta,\tau)=\frac{H_1(\zeta,\tau)}{H_2(\zeta,\tau)},\qquad \zeta, \tau\in(0,-\frac 1 \lambda),\ \zeta<\tau.$$
By continuity, if  we can prove that as $\zeta\to 0$, $H(\zeta, 2\zeta)\to\infty$, and as $\epsilon\to 0$, $H(\frac{1-2\epsilon}{-\lambda},\frac{1-\epsilon}{-\lambda})\to 0$, then the claim holds. But as $\zeta\to 0$ and (hence) $(1+\lambda\zeta)\to 1$,  one can easily estimate to get $H_1(\zeta,2\zeta)=O(\zeta^{2n})$ and $H_2(\zeta,2\zeta)=O(\zeta^{(2n+1)})$, and hence $H(\zeta,2\zeta)=O(\zeta^{-1})$. On the other hand,
as $\epsilon\to 0$, $H_1(\frac{1-2\epsilon}{-\lambda},\frac{1-\epsilon}{-\lambda})=O(\epsilon^{2m+1})$, and $H_2(\frac{1-2\epsilon}{-\lambda},\frac{1-\epsilon}{-\lambda})=O(\epsilon^{2m})$ and hence $H(\frac{1-2\epsilon}{-\lambda},\frac{1-\epsilon}{-\lambda})=O(\epsilon)$.
\end{proof}

According to the claim, we  have $P(b)=P'(b)=0$. The condition  $P(\tau)>0$ for $\tau\in (a,b)$ is automatically satisfied. For if there exists a point $\xi\in(a,b)$ such that $P(\xi)=0$, equation $P'''(\tau)=\kappa(\tau)Q(\tau)=0$ has three roots in $(a,b)$. This is impossible.

From the proof of Proposition \ref{pop}, we see that if $\kappa(a)>0$, then the metric has the PMY asymptotic property. Since $P(a)=P'(a)=0$, $\kappa(a)=0$ is equivalent to $P''(a)=0$, and hence  $P'''(\tau)=0$ has three roots in $(a,b)$. It is impossible.
\end{proof}

\begin{proof}[Proof of Theorem \ref{thmlag0}]
It follows from Propositions \ref{prop3}, \ref{pop}, \ref{poop}, \ref{prpmy} and \ref{ppppp}.
\end{proof}

\section{CscK PMY metrics on $\mathbb P(E\oplus \mathcal O )-M$}\label{SPEC}
Recall that $\mathbb P (E\oplus\mathcal O)$ can be viewed as a compactification of $E$: $E$ can be imbedded into $\mathbb P(E\oplus \mathcal O)$. In fact, let $(U,z=(z_1,\cdots,z_m))$ be a local holomorphic chart of $M$ such that $E|_U$ is (holomorphically) isomorphic to $U\times \mathbb C^n$. If we denote the coordinates of $\mathbb C^n$ as $w=(w_1,\cdots,w_n)$, the imbedding map can be defined as follows: for any $p\in E|_q$, $q\in U$,
$$p\mapsto (q, w_1(p),\cdots, w_n(q))\mapsto (q, [1,w_1(p),\cdots, w_n(p)]).$$
 This map is clearly well-defined on $E$. It defines a section $s$ of $\mathbb P(E\oplus\mathcal O)$:
 $$q\mapsto (q,(0,\cdots, 0))\mapsto (q,[1,0\cdots, 0])$$
which is just the zero section of $E$. Hence we still denote $s(M)$ simply by $M$. Set $D_{\infty}=\mathbb P(E\oplus \mathcal O)-E$. $D_\infty$ is a divisor on $\mathbb P(E\oplus\mathcal O)$ and is called the infinity divisor. By these notations,
$E-M$ is bi-holomorphic to $\mathbb P(E\oplus \mathcal O)-s(M)-D_\infty$. Now the question is when the metric $\omega$ defined on $E-M$ as the above section can be extended across $D_\infty$.

First note that if $\omega$ can be extended across $D_\infty$, $\omega$ must be defined on $E-M$ and is not complete at infinity. Hence according to the proof of Lemma \ref{070501}, the endpoint $b$ of $I=(a,b)$  is finite.

\begin{lemma}\label{lem4}
Let $\omega$ be the bundle adapted metric with momentum profile $\varphi(t)$ in (\ref{70401}). Assuming that there is a constant $b$ such that $\varphi(\tau)$ is positive on 	$(a,b)$ and $\varphi(b)=0$. Then $\omega$ defined on $E-M$ can be extended across $D_\infty$ if and only if $\varphi'(b)=-1$.
\end{lemma}

\begin{proof}
The proof of this lemma is well-known. One can  consult references \cite{Ca2,HS, ACGT}. Here we write out details.

Since the metric $\omega$ on $E-M$ is bundle adapt, we only need to prove that the metric $\omega_0=i\partial\bar\partial f(\tau)$ defined on fiber $E|_q-\pi(q)=\mathbb C^n-0$ can be extended to $\mathbb CP^{n-1}$.

First recall that   $\mathbb CP^n\setminus [1,0,\cdots,0] $ is bi-holomorphic to the line bundle $\mathcal O(1)$ over $\mathbb CP^{n-1}$. Let  $[v_0,\cdots, v_n]$  be the homogeneous  coordinates of $\mathbb CP^{n}$. Here the open set $\mathbb C^n$ is $v_0=1$ and the hyperplane $\mathbb CP^{n-1}$ is $v_0=0$. Hence $w_\alpha=\frac{v_\alpha}{v_0}$ is the $\alpha-$th coordinate of $\mathbb C^n$ and  $[v_1,\cdots,v_n]$ is the  homogeneous coordinates of $\mathbb CP^{n-1}$.
Let $U_\alpha=\{[v_1,\cdots, v_n]\ |\ v_\alpha\not=0\}$ and define $w_\beta^\alpha=\frac{v_\beta}{v_\alpha}$ with $\beta\not=\alpha$. Let $v^{\alpha}$ be the coordinate of the trivialization of $\mathcal O(1)|_{U_\alpha}$. Then its transition function defined on $U_\alpha\cap U_\beta$ is
$$ v^\alpha=\frac{1}{w_\beta^\alpha}v^\beta=w_\alpha^\beta v^\beta.$$
The bi-holomorphic map $\psi:\mathbb CP^n-[1,0,\cdots,0]\to \mathcal O(1)$ is
$$[v_0,\cdots,v_n]\mapsto [w^\alpha_1,\cdots,w^\alpha_{\beta-1},w^\alpha_{\beta+1},\cdots, w^\alpha_n,\frac{v_0}{v_\alpha}],\qquad \text{for}\ v_\alpha\not=0.$$
Define on $\mathcal O(1)$ the function
$$\tilde r^2=\frac{|v^\alpha|^2}{1+\sum_{\beta\not=\alpha}|w_\beta^\alpha|^2}.$$
We have $\tilde r^2=\frac 1 {r^2}$ on $\mathbb C^n-0$.

By Direct computation, we have
 \begin{equation*}
 \begin{aligned}
 \omega_0=& -\sqrt{-1}f'(t)\partial\bar\partial \log\tilde r^2+\sqrt{-1}f''(t)\partial\log \tilde r^2\wedge \bar\partial\log \tilde r^2   \\
 =&(f'(t)+f''(t))\omega_{FS}+\frac{f''(t^2)}{\tilde r^2}(\sqrt{-1}\partial\bar\partial\tilde r^2)\\
 =&(\tau+\varphi(\tau))\omega_{FS}+\frac{\varphi(\tau)}{\tilde r^2}(\sqrt{-1}\partial\bar\partial\tilde r^2)
 \end{aligned}
 \end{equation*}
Define the functions
$$ f_1(\tilde r^2)=\left\{\begin{array}{ll}
\tau+\varphi(\tau), & \tilde r^2>0\\
b& \tilde r^2=0\end{array},
\right.
$$
and
\begin{equation*}
f_2(\tilde r^2)=\left\{\begin{array}{ll}
\frac{\varphi(\tau)}{\tilde r^2}& \tilde r^2>0\\
1& \tilde r^2=0.
\end{array}
\right.
\end{equation*}
Since as $\tilde r^2\to 0$, $\tau\to b$ and $\lim_{\tau\to b}\varphi(\tau)=0$, the function $f_1(\tilde r^2)$ is continuous at $\tilde r^2=0$. As to $f_2$, we
shall  prove that if we take a suitable constant in (\ref{ttau}), then it is also continues at $\tilde r^2=0$.

In fact $f_2(\tilde r^2)$ is smooth. Since $\varphi(b)=0$, $\varphi'(b)=-1$, and $\varphi$ is rational, we can write
\begin{equation*}
\varphi(\tau)=(b-\tau)(1+(b-\tau)\varphi_1(\tau))
\end{equation*}
 for some smooth function $\varphi_1(\tau)$. Then
$$t=\int \frac{1}{\varphi(\tau)}d\tau=-\log (b-\tau)-\varphi_2(\tau).$$
Here $\varphi_2(\tau)$ is a smooth function with $\varphi_2(b)=0$.
Hence
\begin{equation}\label{hh}
\tilde r^2=\frac{1}{r^2}=e^{-t}=e^{\varphi_2(\tau)}(b-\tau)
\end{equation}
and
\begin{equation*}
  \begin{split}
  f_2(\tilde r^2)=(1+(b-\tau)\varphi_1(\tau))e^{-\varphi_2(\tau)}
\end{split}
\end{equation*}
is a smooth function of $\tau$. Moreover, by the implicit function theorem, we can solve (\ref{hh}) to get a smooth function $\tau=\tau(\tilde r^2)$. Hence $f_2(\tilde r^2)$ is a smooth function of $\tilde r^2$.
Now we can also see that $f_1(\tilde r^2 )$ is  smooth at $\tilde r^2=0$ since
$$f_1(\tilde r^2)=\tau +f_2(\tau)\tilde r^2.$$

The metric $\omega_0$ can be extended across $\mathbb CP^{n-1}$ by defining
$$\tilde \omega_0=f_1(\tilde r^2)\omega_{FS}+f_2(\tilde r^2)(\sqrt{-1}\partial\bar\partial\tilde r^2)$$
Since $d(\omega_0)=0$ and $f_1(\tilde r^2)$ and $f_2(\tilde r^2)$ is smooth at $\tilde r^2=0$, $\tilde \omega_0$ is also K\"ahler.
\end{proof}

According to Lemma \ref{lem4}, the momentum profile $\varphi$ in (\ref{70401}) gives a complete cscK metric on  $\mathbb P(E\oplus \mathcal O)- M$ if and only if
 \begin{enumerate}
  \item[(i)] $\varphi(b)=0$ and $\varphi'(b)=-1$ with $b>a$,
   \item[(ii)]$\varphi(\tau)$ is positive on domain $(a,b)$.
 \end{enumerate}

However, condition (ii) is satisfied  automatically if condition (i) holds. For one  can show that under condition (i) $a$ is the unique solution of $P(x)=0$ for $x\in (0,b)$.
In fact, we have the following result.
\begin{lemma}\label{lem0}
If $\varphi(b)=0$ and $\varphi'(b)=-1$, then $\varphi(\tau)>0$ on domains $(0,a)$ and $(a,b)$.
\end{lemma}
\begin{proof}
Since $P''(\tau)=\kappa(\tau)=0$ has at most two roots and we have already $P(a)=P(b)=P'(a)=0$, $P(\tau)=0$ has at most one root $\xi$ in $(0,a)$ or in $(a,b)$.
Also since  $\varphi'(b)=-1$ and $\varphi(b)=0$, $P(\tau)$ is positive as $\tau\to b$. Hence if $\xi\in(a,b)$, there are two cases should be considered. One is $\varphi(\tau)> 0$ on $(a,\xi)\cup(\xi,b)$, the other is $\varphi(\tau)<0$ on $(a,\xi)$. The former is impossible as $\varphi'(\xi)=0$ which would lead to $\kappa(\tau)=0$ has at least three roots. The latter is also impossible as from this one can derive $\varphi''(a)=0$ which would also lead the contradiction. Thus, $\xi\notin (a,b)$.

Since $\varphi(\tau)>0$ when $\tau$ is  near zero, as the same reason,  $\xi\notin (0,a)$.
\end{proof}

Hence,  in the following we only need to find a constant $b$ such that  condition (i) is satisfied. We can solve constants $c_M$ and $c$ from $\varphi(b)=0$ and $\varphi'(b)=-1$ as:
\begin{equation}\label{cM6}
  c_M=\frac{n(n-1)H_1+ L_1 }{H_2}
\end{equation}
and
\begin{equation*}
  c=\frac{n(n-1)H_3+L_2}{H_2}
\end{equation*}
with the definitions (\ref{HHH}) and of
$$
\begin{aligned}
&L_1=Q(b)\int_a^bxQ(x)dx-bQ(b)\int_a^bQ(x)dx,\\
&L_2=bQ(b)\int_a^bQ(x)dx- Q(b)\int_a^b\frac{xQ(x)}{1+\lambda x}dx.
\end{aligned}
$$
So we should determine the range of $c_M$ such that $\varphi(\tau)$  satisfies (i).

 \begin{proposition}\label{ppp}
If $\lambda<0$, the range of   $c_M$ is $\mathbb R$.
 \end{proposition}
 \begin{proof}
Let $\tilde{H}=\frac{n(n-1)H_1+ L_1 }{H_2}$.  First, we take $b=2a$ and  estimate  $\tilde H(a,2a)$ as $a\to 0^+$.  We get
$$(n(n-1)H_1+L_1)(a,2a)= \alpha_1 a^{2n}+O(a^{2n+1})$$
with
$$  \alpha_1=\frac{-2^{n-1}(n+1)^2+n}{n(n+1)}<0,$$
 and $H_2(a,2a)=O(a^{2n+1})$. Since when $\lambda<0$, $H_2(a,2a)>0$, we have
 $$\lim_{a\to 0^+}H(a,2a)=-\infty.$$

On the other hand, we take $b=\sqrt{a}$ and do estimates. As $a\to 0^+$, we also have
$$(n(n-1)H_1+L_1)(a,\sqrt{a})=\alpha_2 a^{n+\frac 12 }+O(a^{n+1}) \qquad \textup{with}\ \ \alpha_2=\frac{-2\lambda m}{n(n+1)(n+2)}>0,$$
and $H_2(a,\sqrt{a})=O(a^{n+1})$. Hence we have $\lim_{a\to 0^+}\tilde H(a,\sqrt{a})=+\infty$.

Now the result follows from the  continuity of $\tilde H$.
\end{proof}

\begin{proposition}\label{pppp}
If $\lambda>0$, the range of $c_M$ is $(m(m+2n-1)\lambda,\infty)$.
\end{proposition}
\begin{proof}
First we note that when $\lambda>0$, $H_2(a,2a)<0$. Thus by the estimates in the proof of the above lemma, we get $\lim_{a\to 0^+}\tilde H_2(a,2a)\to +\infty$.

Next we do estimates: as $b\to +\infty$��
\begin{align*}
&  (n(n-1)H_1+L_1)(a,b)\sim \frac{-m(m+2n-1)}{(m+n)^4-(m+n)^2}b^{2(m+n)}\\
 & H_2\sim\frac{-\lambda^{2m-1}b^{2m+2n}}{(m+n)^2(m+n-1)(m+n+1)}
\end{align*}
Hence
\begin{equation*}
  \lim _{b\to \infty} c_M=m(m+2n-1)\lambda
\end{equation*}

At last, we need  to prove
$$c_M>m(m+n-1)\lambda$$
i.e., to prove when $b\geq a$,
$$K(b)=n(n-1)H_1(b)+L_1(b)-m(m+2n-1)\lambda H_2(b)<0.$$
This is a calculus exercise and we leave to readers.

In summary, the range of $c_M$ is $(m(m+2n-1)\lambda,\infty)$.
\end{proof}

\begin{proposition}\label{ooo}
The metrics in Propositions \ref{ppp} and \ref{pppp} admit the  PMY asymptotic property.
\end{proposition}
\begin{proof}
We need to prove $\kappa(a)>0$. By Lemma \ref{lem0},  $a$ is a minimum of $\varphi(x)$. From the Taylor expansion of $\varphi(x)$ at $x=a$, if $\kappa(a)=0$,  $\kappa'(a)=-\frac{\lambda c_M}{(1+\lambda a)^2}-\frac{n(n-1)}{a^2}=0$. So if $\lambda c_M\geq 0$, it is impossible. Thus for $\lambda>0$ (hence $c_M>0$), or $\lambda<0$ and $c_M\leq 0$, the metrics are PMY. For the case  $\lambda<0$ and $c_M>0$, if $k'(a)=0$, then
  $$\kappa''(x)=\frac{2\lambda^2 c_M}{(1+\lambda x)^3}+\frac{2n(n-1)}{x^3}$$
  would have one root in   $ (0,b)$ when  $\varphi(a)=\varphi'(a)=\varphi''(a)(=\kappa(a))=\varphi(b)=0$. But it is impossible when $c_M>0$. Hence $\kappa'(a)\not=0$ and then $\kappa(a)>0$.   The metrics are also PMY.
\end{proof}

\begin{proof}[The proof of Theorem \ref{thmpm}]
It follows from Propositions \ref{ppp}, \ref{pppp}, and \ref{ooo}.
\end{proof}

We give an example of cscK metric with $\lambda>0$.
 \begin{example}
   Let $m=1$, $n=2$, $\lambda=1$ and $a=1$. Then
   \begin{align*}
    \varphi(\tau)&=\frac{1}{\tau(\tau+1)}\int_1^{\tau}(\tau-x)(c_M x+2(1+x)-cx(1+x))dx\\
    &=-\frac{1}{12\tau(\tau+1)} (-1 + \tau)^2 \biggl(c \tau^2 +(4c-2c_M-4)\tau +7c-4c_M -20 \biggr)
   \end{align*}
We can solve  $\varphi(b)=\varphi'(b)=0$ to get
 $$c_M=\frac{2 (-13 + 37 b + 39 b^2 + 7 b^3 + 2 b^4)}{(-1 + b)^2 (1 + 4 b + b^2)}.$$
For example, let $b=2$. Then
$$c_M=\frac{610}{13}\qquad \text{and}\qquad   c =\frac{276}{13}$$
which implies
$$\varphi(x)=\frac{-23 x^4 +60 x^3+13 x^2 -114 x +64}{13x(1+x)}.$$
We give a picture of $\varphi(\tau)$ as Figure \ref{P5}.
 \begin{figure}[H]
\centering
  \includegraphics[width=0.6\textwidth]{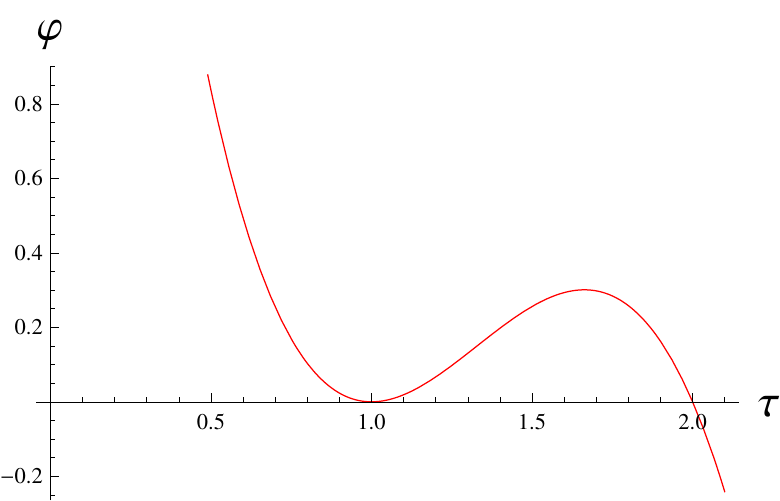}
  \caption{\small \color{blue} the graph of $\varphi(\tau)$ on $[1,2]$ with $\lambda=1$, $c_M=\frac{610}{13}$ and  $c =\frac{276}{13}$.}\label{P5}
\end{figure}
 \end{example}

 We also give two examples of csck metrics with $\lambda<0$.
 \begin{example}
 Let $m=1$, $n=2$, $\lambda=-1$ and $c_M=2$. We choose  $a=0.001$, then the graph as Figure \ref{P6} shows $b\simeq 0.0893745$ and $c\simeq 68.7366$, or $b\simeq 0.998$ and $c\simeq 11.9761$.  We give a picture of $\varphi(\tau)$ with $b=0.0894$ as Figure \ref{P7}.
  \begin{figure}[ht]
  	\includegraphics[width=0.5\textwidth]{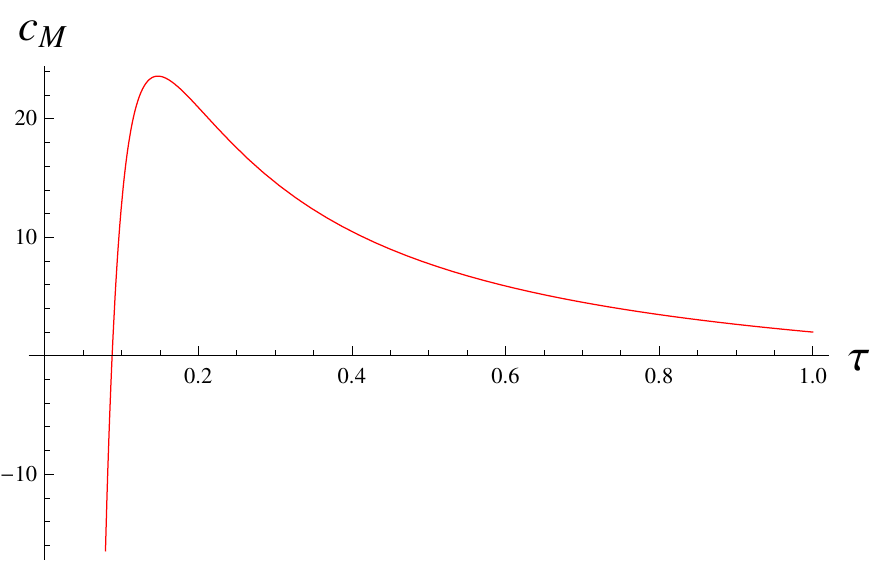}
  	\caption{\small \color{blue} the graph of $c_M$  with $a=0.001$.} \label{P6}
  	\end{figure}	
\begin{figure}[ht]
  	\includegraphics[width=0.5\textwidth]{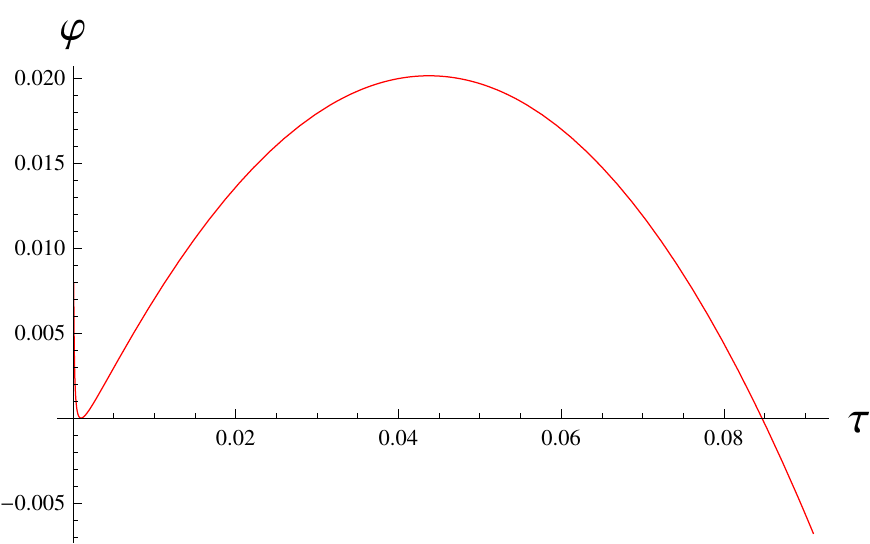}
  	\caption{\small \color{blue} the graph of $\varphi(\tau)$  with $a=0.001$, $\lambda=-1$ and $c_M=2$.}\label{P7}
  	\end{figure}

\end{example}

\begin{example}
Let $m=1$, $n=2$, $\lambda=-1$ and $c_M=-2$. If $a=0.1$, then $b\simeq 0.61146$ and $c\simeq 5.02242$. We give a picture of $\varphi(\tau)$ as Figure \ref{P8}.
   \begin{figure}[h]\centering
  \includegraphics[width=0.5\textwidth]{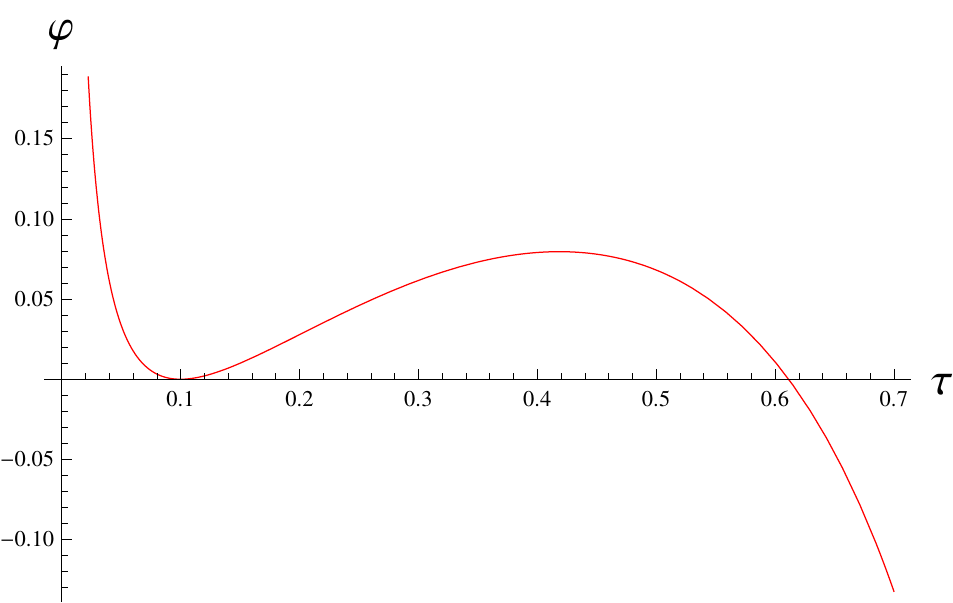}
  \caption{\small \color{blue}  the graph of $\varphi(\tau)$  with $a=0.1$,  $\lambda=-1$ and  $c_M=-2$.}\label{P8}
\end{figure}
\end{example}


\begin{thebibliography}{99}
\bibitem{ACGT}V. Apostolov, D. Calderbank, P. Gauduchon and C. T{\o}nnesen-Friedman; {\sl Extremal K\"ahler metrics on ruled manifolds and stability.} Ast��risque {\bf322} (2008), 93--150.
\bibitem{AP1} C. Arezzo and F. Pacard; {\sl Blowing up and desingularizing constant scalar curvature K\"ahler manifolds.} Acta. Math. {\bf196} (2006), 179--228.
\bibitem{AP2}  C. Arezzo and F. Pacard; {\sl Blowing up K\"ahler manifolds with constant scalar curvature. II.} Ann. of Math. (2) {\bf 170} (2009), 685--738.
\bibitem{APS}C. Arezzo, F. Pacard and M. Singer,  {\sl Extremal metrics on blow ups.} Duke Math. J.  {\bf 157} (2011), 1--51.

\bibitem{Ca1}E. Calabi; {\sl M\'etriques, K\"ahleriennes et fibr\'es holomorphes.} Annales Scientifiques de l'\'Ecole Normale Sup\'erieure {\bf 12} (1979), 268--294.
\bibitem{Ca2} E. Calabi; {\sl Extremal K\"ahler metrics, in Seminar on Differential Geometry.} Ann. of Math. Stud. {\bf 102}, 259--290, Princeton Univ. Press, Princeton, 1982.
\bibitem{FYZ1}J. Fu, S.--T. Yau and  W. Zhou; {\sl  Complete cscK metrics on the local models of the conifold transition.} Commun. Math. Phys. {\bf 335} (2015), 1215--1233.
\bibitem{FYZ2}J. Fu, P. Gao and  W. Zhou; Nonexistence for complete canonical metrics  and the Calabi--Futaki invariant for K\"ahler metrics of Poincar\'e-Mok-Yau asymptotic property on complete K\"ahler manifolds. Preprint.
\bibitem{HS}A. D. Hwang and M. A. Singer; {\sl A momentum construction for circle--invariant K\"ahler metrics.} Trans. Amer. Math. Soc. {\bf 354} (2002) 2285--2325.
\bibitem{KS} N. Koiso and Y. Sakane; {\sl Non-homogeneous K\"ahler--Einstein metrics on compact complex manifolds in Curvature and topology of Riemannian manifolds.} Springer LNM {\bf V. 1201}, 1986, 165--179.
\bibitem{MY}  N. Mok and S.--T. Yau; {\sl Completeness of the  K\"ahler--Einstein metric on bound domains and the Characterization of domain of holomorphy by curvature condition.} Proceddings of Symposia in Pure mathematics, {\bf V.39} (1983), Part 1, 41--59.
  \bibitem{PP} H. Pedersen and Y.--S. Poon; {\sl Hamiltonian constructions of K\"ahler--Einstein metrics and K\"ahler merics of constant scalar curvature.} Commun. Math. Phys. {\bf 136} (1991), 309--326.
  \bibitem{Si}S. Simanca; {\sl  K\"ahler metrics of constant scalar curvature on bundles over $\mathbb CP^{n-1}$.} Math. Ann.  {\bf 291} (1991), 239-246.
  \bibitem{Ga} G. Sz\'ekelyhidi; {\sl Blowing up extremal K\"ahler manifolds.} Duke Math. J. {\bf 161} (2012), 1411--1453.
  \bibitem{Ga2}G. Sz\'ekelyhidi; {\sl Blowing up extremal K\"ahler manifolds II.} Invent. Math.  {\bf 200}  (2015), 925--977
\bibitem{TY1}  G. Tian and S.--T. Yau; {\sl Complete K\"ahler manifolds with zero Ricci curvature. I.} J. Amer. Math. Soc. {\bf 3} (1990), 579--609.
\bibitem{TY2}  G. Tian and S.--T. Yau; {\sl  Complete  K\"ahler  manifolds with zero Ricci curvature. II.} Invent. Math. {\bf 106} (1991), 27--60.
\bibitem{Yau} S.--T. Yau; {\sl On the Ricci curvature of a compact  K\"ahler manifold and the complex Monge--Amp\`ere equation. I.} Comm. Pure. Appl. Math. {\bf 31} (1978), 339--411.
 \bibitem{Yau2}   S.--T. Yau, {\sl A general Schwarz lemma for K\"ahler manifolds.} Amer. J. Math {\bf 100} (1978), 197--203.

\end{thebibliography}
\end{document}